\documentclass[a4paper,12pt]{article}
\usepackage[a4paper, total={7in, 10in}]{geometry}

\usepackage{amsmath, amsthm, amssymb, hyperref, color}
\usepackage{MnSymbol} % for \dashedrightarrow command
\usepackage{graphicx}
\usepackage{caption}
\usepackage[all]{xypic}
\usepackage{verbatim}
\usepackage{chemfig, chemnum}
\usepackage{tikz}
\usepackage[linesnumbered,lined,commentsnumbered]{algorithm2e}
\usepackage{caption}
\usepackage{subcaption}
\usepackage{float}
\usepackage{tikz-cd}
\usetikzlibrary{calc}
\usepackage{extarrows}

\usepackage{makecell}
\usepackage{array}
\usepackage{fancyvrb}
\usepackage{enumitem}
\usepackage{standalone}
\usepackage{subcaption}

\usepackage[T1]{fontenc}
\usepackage{imakeidx}
\makeindex

\usepackage[backend=biber,style=numeric,  maxnames=99]{biblatex}
\newcommand{\bR}{\mathbb R}

\newcommand{\bP}{\mathbb P}
\newcommand{\cA}{\mathcal A}

\renewcommand{\cH}{\mathcal H}

\renewcommand{\phi}{\varphi}

\newcommand{\conv}{{\rm conv}}

\newcommand{\cl}{{\rm cl}}
\newcommand{\inter}{{\rm int}}

\newcommand{\gen}[1]{\langle {#1} \rangle}

\newcommand{\Gr}{{\rm Gr}}

\newcommand{\GL}{{\rm GL}}
\newcommand{\PGL}{{\rm PGL}}

\usepackage{amsthm}

\theoremstyle{remark}
\newtheorem*{remark}{Remark}

\theoremstyle{plain}
\newtheorem{theorem}{Theorem}[section]
\newtheorem{lemma}[theorem]{Lemma} 
\newtheorem{proposition}[theorem]{Proposition}

\newtheorem{cor}[theorem]{Corollary}

\makeatletter
\newtheorem*{rep@theorem}{\rep@title}
\newcommand{\newreptheorem}[2]{%
\newenvironment{rep#1}[1]{%
 \def\rep@title{#2 \ref{##1}}%
 \begin{rep@theorem}}%
 {\end{rep@theorem}}}
\makeatother
\newreptheorem{theorem}{Theorem}
\newreptheorem{cor}{Corollary}
\newreptheorem{proposition}{Proposition}

\theoremstyle{definition}
\newtheorem{definition}[theorem]{Definition}
\newtheorem{eg}[theorem]{Example}

\newtheorem{question}[theorem]{Question}

\title{ Exterior Cyclic Polytopes and Convexity of Amplituhedra }

\author{Elia Mazzucchelli\footnote{Max-Plank-Institut f\"ur Physik, Werner-Heisenberg-Institut,
D-85748 M\"unchen, Germany \\ \texttt{email: eliam@mpp.mpg.de}} \ and Elizabeth Pratt\footnote{University of California, Berkeley
 \\ \texttt{email: epratt@berkeley.edu}}}

\begin{document}

\maketitle
\thispagestyle{empty}

\begin{abstract}
The amplituhedron is a semialgebraic set in the Grassmannian. We study convexity and duality of amplituhedra. We introduce a notion of convexity, called \textit{extendable convexity}, for real semialgebraic sets in any embedded projective variety. We show that the $k=m=2$ amplituhedron is extendably convex in the Grassmannian of lines in projective three-space. In the process we introduce a new polytope called the \emph{exterior cyclic polytope}, generalizing the cyclic polytope. It is equal to the convex hull of the amplituhedron in the Plücker embedding. We undertake a combinatorial analysis of the exterior cyclic polytope, its facets, and its dual. Finally, we introduce the \textit{(extendable) dual amplituhedron}, which is closely related to the dual of the exterior cyclic polytope. We show that the dual amplituhedron for $k=m=2$ is again an amplituhedron, where the external matrix data is changed by the twist map.

%a class of polytopes, which we call exterior cyclic polytopes. The latter have a nice combinatorial description....
\end{abstract}

\clearpage
\thispagestyle{empty}

\tableofcontents

\clearpage
\setcounter{page}{1}

\section{Introduction}
\label{sec:Introduction}

The real Grassmannian $\Gr(k,n)$ is a projective variety of dimension $k(n-k)$ embedded in real projective space $\bP_\bR^{N}$ with $N=\binom{n}{k}-1$. Its points parameterize $k$-dimensional subspaces of $\bR^n.$ Each such point can be represented (non-uniquely) as the rowspan of a $k \times n$ matrix, whose $k \times k$ minors are called the \emph{Pl\"ucker coordinates} of the point.

Recently, certain semialgebraic subsets of the Grassmannian have become relevant to high-energy physics for the computation of scattering amplitudes. The most basic among these is the \emph{non-negative Grassmannian} $\Gr_{\geq 0}(k,n),$ defined as the locus in $\Gr(k,n)$ where the Pl\"ucker coordinates are non-negative. In other words, it is the intersection of $\Gr(k,n)$ with the non-negative orthant in~$\bP_\bR^{N}$. More generally, fix $k, m, n$ with $n \geq k+m$ and a real $(k+m) \times n$ matrix $Z$ with positive maximal minors. The \emph{amplituhedron} $\cA_{k,m,n}(Z)$ for $n \geq k+m$ is the image of the linear projection
\begin{equation}
\label{Z_tilde_map}
\begin{aligned}
    \wedge^kZ \, : \, \Gr_{\geq 0}(k,n) & \to \Gr(k,k+m) \\
    [M] & \mapsto [Z^TM] \, ,
\end{aligned}
\end{equation}
where $[M]$ denotes the point represented by the $k \times n$ matrix $M$. The special case $m=4$ is the one relevant to physics \cite{Arkani_Hamed_2014}. %More generally, one defines a \textit{Grassmann polytope} \cite{Lam:2014omj} as the image of a positroid cell under \eqref{Z_tilde_map}, and a \textit{Grasstope} as the image of the non-negative Grassmannian under \eqref{Z_tilde_map} without the assumption on $Z$ being positive \cite{Arkani_Hamed_2017,mandelshtam2025combinatoricsm1grasstopes}.

\begin{eg}[The case $k=1$]
Here the non-negative Grassmannian is simply the non-negative orthant in $\bP_\bR^{n-1},$ which may be seen as a simplex in the affine chart where the sum of the coordinates is nonzero. Let $Z_1, \, \ldots, \, Z_n \in \mathbb{R}^{m+1}$ denote the columns of $Z$, viewed as points in $\mathbb{P}^m_\mathbb{R}$. The amplituhedron $\mathcal{A}_{1,m,n}(Z) \subset \bP_\bR^m$ is the convex hull of the $n$ points $Z_1, \, \ldots, \, Z_n.$ Since $Z$ is totally positive, the amplituhedron has the combinatorial type of a cyclic polytope.
\end{eg}

For $n=k+m$, the amplituhedron is isomorphic to the non-negative Grassmannian ${\Gr_{\geq 0}(k,k+m)}$, so it is equal to the intersection of $\Gr(k,k+m)$ with the standard simplex in the ambient projective space. In this paper we study whether more general amplituhedra may be described as the intersection of the Grassmannian with a convex polytope. 
With this motivation, we define the \emph{exterior cyclic polytope}
$C_{k,m,n}(Z)$ as the projection of the non-negative orthant in $\bP_\bR^{N} \cong \mathbb{P}(\bigwedge^k \mathbb{R}^n)$ under the linear map $\wedge^kZ \, \colon \, \mathbb{P}(\bigwedge^k \mathbb{R}^n) \dashrightarrow \mathbb{P}(\bigwedge^k \mathbb{R}^{k+m})$. The case $k = 1$ recovers the cyclic polytope. We prove the following result for $k=m=2$.

\begin{reptheorem}{thm:int_k=m=2}
We have that
\begin{equation}
    \mathcal{A}_{2,2,n}(Z) = \Gr(2,4) \cap C_{2,2,n}(Z) \, ,
\end{equation}
for every real $4 \times n$ matrix $Z$ with positive maximal minors.
\end{reptheorem}

% Corollary~\ref{cor:m=k=2_conv_hull}
%Thus the amplituhedron $\cA_{2,2,n}(Z)$ is extendably convex for every positive $4 \times n$ matrix $Z.$

More generally, we define a set in the Grassmannian to be \emph{extendably convex} if it is the intersection of the Grassmannian with a convex body in the ambient Pl\"ucker space. This notion of convexity was first considered by Busemann~\cite{busemann1961convexity}. In this language, we obtain that the amplituhedron $\cA_{2,2,n}(Z)$ is extendably convex for every positive $4 \times n$ matrix~$Z.$

% We provide insight for the analogous result of Theorem~2 for $k=2$ and $m=n-4$. In particular, in Proposition~\ref{prop:boundcapamp} we determine the intersection of the amplituhedron with the exterior cyclic polytope's boundary, therefore characterizing a set of linear boundaries of the amplituhedron itself for $k=2$ and $m=n-4$. This result relies on a nice duality between the linear matroid of $\wedge^2 Z$ and Kalai's \emph{hyperconnectivity matroid}~\cite{brakensiek2024rigidity}.

Furthermore, we uncover a relationship between the dual of the exterior cyclic polytope and the \emph{twist} $\tau(Z)$ of the matrix $Z$. The twist map $\tau$ has been studied extensively by both mathematicians \cite{BFZ,Galashin_2020,Muller_2017} and physicists \cite{Arkani_Hamed_2018} in relation to parity duality. In the special case $k=2$ and $m=2,$ the physical interpretation of the twist map is that it exchanges the maximal-helicity-violating (MHV) sector of the one-loop amplitude with its parity conjugate ($\overline{\text{MHV}}$). % In physics, and in particular in the context of scattering amplitudes in $\mathcal{N}=4$ super Yang-Mills, it corresponds to the action $k \rightarrow n-k-4$ on the helicity of the particle, arising from the action of the parity operator in the spacetime picture. 
\begin{repproposition}{prop:parity dual}
    We have that
\begin{equation}
    \widetilde{C}_{2,2,n}(Z) =  C_{2,2,n}(W)^* \, ,
\end{equation}
where $W=\tau(Z)$ is the twist map applied to $Z$, according to Definition~\ref{def:twist_map}. 
\end{repproposition}
Here the polytope $\widetilde{C}_{k,m,n}(Z)$ is obtained from the polytope $C_{k,m,n}(Z)$ by keeping only those facet inequalities such that the supporting hyperplane is a Schubert divisor when restricted to $\Gr(k,k+m)$. We call these \emph{Schubert facets}. We prove Proposition \ref{prop:parity dual} by studying the linear matroid of $\wedge^k Z.$  For the case $k=2,$ this is equal to the dual of the \emph{hyperconnectivity matroid}, which was introduced by Kalai in \cite{kalai} to study highly connected graphs and has been studied more recently in the context of graph rigidity \cite{ brakensiek2024rigidity, Crespo_Ruiz2023}. We characterize Schubert facets for $k=m=2.$

\begin{reptheorem}{thm:nonschubert}
    The Schubert facets of the polytope $C_{2,2,n}(Z)$ are exactly the $\binom{n}{2}$ hyperplanes given by the vanishing of $\gen{Y \, \bar{i} \bar{j}}$ for $1 \leq i<j \leq n,$ where $\bar{i} \bar{j} := (i-1 i i+1) \cap (j-1 j j+1).$ Furthermore, these Schubert facets intersect transversally in $\Gr(2,4)$ for every $Z \in {\rm Mat}_{>0}(4,n)$.
\end{reptheorem}

The connection between the amplituhedron and physics arises in the context of positive geometries~\cite{Arkani_Hamed_2017}. A positive geometry is a triple $(X, X_{\geq 0}, \Omega)$ consisting of a complex algebraic variety $X$ defined over $\bR,$ a semialgebraic subset $X_{\geq 0}$ of its real points, and a complex meromorphic form on $X$ called the \emph{canonical form}. This form has simple poles on the algebraic boundary of $X_{\geq 0}$ and is holomorphic elsewhere, and its residues along the boundary satisfy an additional recursive property (see \cite[Section 2.1]{Arkani_Hamed_2017}). Amplituhedra describe scattering amplitudes in a specific quantum field theory, called $\mathcal{N}=4$ super Yang-Mills. Here $n$ is the number of scattered particles and $k$ determines the helicity of the particles. The conjecture~\cite{Arkani_Hamed_2014} is that the (tree) amplituhedron $\mathcal{A}_{k,4,n}(Z)$ is a positive geometry and that its canonical form computes the tree-level scattering amplitude.

The case $k=m=2$, while seemingly specific, is the only case in which this conjecture has been proven~\cite{Ranestad_2024}. It coincides with the one-loop amplituhedron for $k=0$ and $m=4$~\cite[Section 6.4]{Arkani_Hamed_2017}. 

Our motivation for studying convexity and duality comes from this context of positive geometries. In many
cases of physical significance the canonical function of the positive geometry has a representation
as an integral over a “dual” geometry with respect to a non-negative measure. In the prototypical example of a polytope in projective space, the canonical function can be expressed as the Laplace transform of the characteristic function on the dual polytope~\cite[Section 7.4.2]{Arkani_Hamed_2017}. The existence of such a representation relies on the convexity of the polytope. The long-standing hope, going back to the influential paper of Hodges~\cite{Hodges:2009hk}, is that such a description also exists for the canonical form of the amplituhedron. This would imply that amplitudes in $\mathcal{N}=4$ super Yang-Mills are volumes of certain geometric objects. Such a integral representation could also potentially imply that amplitudes possess a strong analytic property called complete monotonicity, which has been observed to hold for many functions appearing in quantum field theories~\cite{Henn:2024qwe}. In Section~\ref{sec:Dual amplituhedra} we give a candidate for the underlying semialgebraic set of the dual amplituhedron. This is related to the dual of the exterior cyclic polytope, and we describe it in the case of $k=m=2.$ 
\begin{repcor}{cor:dual_ampl}
    The extendable dual amplituhedron for $k=m=2$ in twistor space is equal to the twisted amplituhedron, which has a semialgebraic description as in eq.~\eqref{semialg_dual_ampl}.
\end{repcor}

\medskip
The outline of the paper is as follows. 
In Section~\ref{sec:Review of twistors and the twist map} we review twistor notation, which we use to denote intersections of linear spaces in projective space by elements in the (projective) exterior algebra. We also recall the definition of the twist map on matrices with positive maximal minors, first introduced in the context of cluster algebras in~\cite{BFZ}.
In Section~\ref{sec:Convexity and intersections with polytopes}, we extend Busemann's definition in \cite{busemann1961convexity} to any semialgebraic set in a real embedded real projective variety. That is, we call a set in such a variety \emph{extendably convex} if it is equal to the intersection of the variety with a convex body in the ambient projective space. We then present a technical lemma, related to the intersection of a polytope with a real embedded projective variety, which we use in the convexity proof in Section~\ref{sec:Convexity of Amplituhedra}. 

In Section~\ref{sec:Exterior Cyclic Polytopes} we define \emph{exterior cyclic polytopes} and explore their combinatorial properties. In Section~\ref{sec:Schubert hyperplanes} we study a special class of their facets cut out by Schubert hyperplanes, which will be useful in Section~\ref{sec:Dual amplituhedra}. In Section~\ref{sec:Convexity of Amplituhedra} we connect exterior cyclic polytopes to amplituhedra. We first prove a result about linear boundaries of the amplituhedron for $k=2$ and $m=n-4$, and then prove that the amplituhedron for $k=m=2$ is equal to the intersection of the exterior cyclic polytope with the Grassmannian. Finally, in Section~\ref{sec:Dual amplituhedra} we define a notion of duality for semialgebraic sets in the Grassmannian, and explore its concrete incarnation for the amplituhedron at $k=m=2$. Appendix~\ref{app:hasse} computes a poset of bases of the wedge power matroid $W_{2,2,n}$ and Appendix~\ref{app:Computations for W_{2,3,n}} computes all circuits of $W_{2,3,n}$, see Definition~\ref{def:wedge_matroid}.

% Since the discovery of the amplituhedron, positive geometries for other quantum field theories have been found~\cite{Herrmann_2022,Arkani-Hamed:2024jbp}. 
 
% Even though amplituhedra have been introduced more than a decade ago, they are still very mysterious objects in general. Recent advances in the understanding of amplituhedra's tiles and tilings have been made \cite{Karp_2017,Parisi:2021oql,Even-Zohar:2023del} for $m=1,2,4$ respectively. In \cite{Ranestad_2024} the authors prove that $\mathcal{A}_{2,2,n}(Z)$ is a positive geometry  for every $n \geq 4$, by studying its adjoint hypersurface, which is defined as the vanishing locus of the numerator of the canonical form. A similar analysis has been carried out also for one case of loop amplituhedra \cite{dian2024twoloopamplituhedron}.

\section{Review of twistors and the twist map}
\label{sec:Review of twistors and the twist map}

%In this section we fix some notation and review a map on positive matrices called parity duality. We first recall some results from multilinear algebra \cite{gelfand1994discriminants}. We identify an $r$-dimensional linear subspace $V$ of $\mathbb{R}^N$ spanned by $r$ linearly independent vectors $v_i \in \mathbb{R}^N$ with the kernel of an alternating form $\omega_V = v_1 \wedge \dots \wedge v_r \in \bigwedge^{r} \mathbb{R}^N$. According to this identification, given $\omega_V = v_1 \wedge \dots \wedge v_r$ and $\omega_W = w_1 \wedge \dots \wedge v_s$ with $r+s \geq N$, intersection $V \cap W$ is identified with
%\begin{equation}
%\begin{aligned}
%    \omega_{V \cap W} &= *((* \omega_V) \wedge (* \omega_W)) = \iota_{*\omega_V} (\omega_W) \\
%    &= \sum_{I \in \binom{[r]}{r+s-N}} \varepsilon(I,I^c) \, \det(v_{I^c}, w_1,\dots, w_s) \, v_I \, ,
%\end{aligned}
%\end{equation}
%where $*$ denotes the Hodge star operator and $\iota$ the interior product on the exterior algebra of $\mathbb{R}^N$. The sign $\varepsilon(I,I^c) = \pm 1$ is equal to the sign of the permutation $(I,I^c)$, where $I^c := [r] \setminus I$. Note that if $\dim(V \cap W) > N -r-s$, then $\omega_{V \cap W} = 0$.

From now on, we work over $\mathbb{R}$. Let $\mathbb{P}^N$ be the real projective space of dimension $N$ and let $\Gr(k,n)$ be the real Grassmannian of $k$-dimensional linear subspaces of $\mathbb{R}^n$.
Let $Z$ be a real $(k+m) \times n$ matrix with positive maximal minors, where $n \geq k+m$. We identify the $i$-th column of $Z$ with a point in $\mathbb{P}^{k+m-1}$, which we denote by $Z_i$ or simply by $i$. We use $(i_1 \, \cdots \, i_r)$ to denote the $(r-1)$-dimensional linear space in $\mathbb{P}^{k+m-1}$ spanned by the points $i_1, \, \dots, \, i_r$ with $1 \leq r \leq k+m$. Then $(i_1 \, \cdots \, i_r) $ may be viewed as an element of $\Gr(r,k+m)$, corresponding to $Z_{i_1}\wedge \dots \wedge Z_{i_r}$ via the Plücker embedding into $\bP(\bigwedge^r \bR^{k+m})$. Thus we can represent the (geometric) intersection of subspaces in $\mathbb{P}^{k+m-1}$ by an element in the exterior algebra $\mathbb{P} (\bigwedge \mathbb{R}^{k+m-1})$. That is, setting $t := N - r - s \geq 1$ we have that %\cite{gelfand1994discriminants}
\begin{equation}
\begin{aligned}
    (i_1 \, \cdots \, i_r) \cap (j_1 \, \cdots \, j_s) &= \sum_{I \in \binom{\{i_1,\dots,i_r\}}{t}} \varepsilon(I,I^c) \, \langle I^c, j_1  \dots  j_s \rangle \, I
    \\
    & =  \, \sum_{J \in \binom{\{j_1,\dots,j_s\}}{t}}  \varepsilon(J,J^c) \, \langle  i_1  \dots  i_r , J^c \rangle \, J \, ,
\end{aligned}
\end{equation}
where $\varepsilon(I,I^c) = \pm 1$ is equal to the sign of the permutation $(I,I^c)$, $I^c := \{1,\dots,r\} \setminus I$ and $\langle a_1 \dots  a_{k+m} \rangle$ denotes the corresponding maximal minor of $Z$. More generally, in the following we use angle brackets to denote the determinant of matrices obtained by stacking together the arguments in the brackets. We also used the shorthand notation $I = (i_{a_1} \, \cdots \, i_{a_t})$ if $I$ is equal to $\{a_1,\dots,a_t \} \in \binom{\{i_1,\dots,i_r\}}{t}$, and similarly for $J$. In the special case where $r+s=N$, we have that
\begin{equation}
    (i_1 \, \cdots \, i_r) \cap (j_1 \, \cdots \, j_s) := \langle  i_1  \dots  i_r  j_1  \dots j_s \rangle \, ,
\end{equation}
which vanishes if and only if the subspaces $ (i_1 \, \cdots \, i_r) $ and $ (j_1 \, \cdots \, j_s)$ intersect non-trivially.

\begin{eg}
Let $k+m=4.$ In $\mathbb{P}^3$, a plane $(i_1 i_2 i_3)$ intersects a line $(j_1j_2)$ in a single point whenever $|\{i_1,i_2,i_3\} \cap \{j_1,j_2\}|\leq 1$, since $Z$ has nonzero maximal minors. This point is given by
\begin{equation}\label{plane_int_line}
\begin{aligned}
    (i_1 i_2 i_3) \cap (j_1j_2) &= i_1 \, \langle i_2 i_3 j_1 j_2 \rangle  -  i_2 \, \langle i_1 i_3 j_1 j_2 \rangle  + i_3 \, \langle i_1 i_2 j_1 j_2 \rangle \\
    &= j_1 \langle i_1i_2i_3 j_2 \rangle -  j_2 \langle i_1i_2i_3 j_1 \rangle \, .
\end{aligned}
\end{equation}
Similarly, two distinct planes $(i_1 i_2 i_3)$ and $(j_1 j_2 j_3)$ intersect in the line
\begin{equation}\label{plane_int_plane}
\begin{aligned}
    (i_1 i_2 i_3) \cap (j_1j_2j_3) &= (i_1  i_2) \, \langle  i_3  j_1 j_2 j_3 \rangle -  (i_1 i_3) \, \langle  i_2 j_1 j_2 j_3 \rangle +  (i_2 i_3) \, \langle  i_1 j_1 j_2 j_3 \rangle  \\
    &= (j_1 j_2) \, \langle   i_1 i_2 i_3 j_3 \rangle -  (j_1 j_3) \, \langle   i_1 i_2 i_3 j_2 \rangle +  (j_2 j_3) \, \langle   i_1 i_2 i_3 j_1 \rangle  \, .
\end{aligned}
\end{equation}
\end{eg}

We now move to the definition of the twist map on matrices. We denote by $\text{Mat}_{> 0}(k+m,n)$ the space of matrices with positive maximal minors, where we assume that $n\geq k+m$. We call an element in $\text{Mat}_{> 0}(k+m,n)$ a \textit{positive matrix}. Let $Z$ be in $ \text{Mat}_{> 0}(k+m,n)$ and denote by $Z_i \in \mathbb{R}^{k+m}$ the $i$-th column of $Z$. We use the \textit{twisted cyclic} ordering on the index set $[n]:=\{1,\dots,n\}$, for which $Z_{n+i} = (-1)^{k-1} \, Z_i$. We identify $\bigwedge^{k+m-1} \mathbb{R}^{k+m}$ with $ \mathbb{R}^{k+m}$ via the standard inner product. Consider then the vector in $\mathbb{R}^{k+m}$ defined by
\begin{equation}\label{eq:ibardef}
        W_i := Z_{i-m+1} \wedge Z_{i-m+2} \wedge \dots \wedge Z_{i} \wedge \dots \wedge Z_{i+k-1} \, , \qquad i \in [n] \, .
\end{equation}

Geometrically, we may think of $W_i$ as the normal to the hyperplane in $\bR^{k+m}$ spanned by the vectors $Z_{i-m+1} , \dots, Z_{i+k-1}$.
\begin{eg}
  Let $k=m=2$. Then the vector $W_i \in \mathbb{P}^3_\mathbb{R}$ is normal to the plane $\bar{i}:=(i-1  i  i+1)$, where the indices are taken modulo the twisted cyclic action described above. Moreover, 
\begin{equation}\label{ij_bar}
    \overline{ij} :=  \bar{i} \cap \bar{j} = W_i \wedge W_j  \ \in \ \mathbb{P}(\bigwedge^2 \mathbb{R}^4) \, .
\end{equation}  
According to the twistor notation, we identify $\overline{ij}$ with the line in $\mathbb{P}^3$ given by the intersection of the two planes $\bar{i}$ and $\bar{j}$. The lines \eqref{ij_bar} will be relevant when discussing the exterior cyclic polytope.
\end{eg}

\begin{definition}[Twist map]\label{def:twist_map}
We now define the \emph{twist map} as
\begin{align*}
    \tau \, : \, \text{Mat}_{> 0}(k+m,n) & \to \text{Mat}(k+m,n) \, , \\
    Z \mapsto W \, ,
\end{align*}
where $W$ is the matrix with columns $W_1, \, \ldots, \, W_n$ as in \eqref{eq:ibardef}.   
\end{definition}
The twist map was first introduced by Bernstein, Fomin, and Zelevinsky~\cite{BFZ} in the context of cluster algebras. Up to sign changes and a cyclic shift, the twist map in Definition~\ref{def:twist_map} coincides with the \textit{right twist} as defined by Muller and Speyer~\cite{Muller_2017}, but we will follow the convention of~\cite[Definition 4.1]{Galashin_2020}. Muller and Speyer also define a left twist and show that the right and left twists are mutual inverse diffeomorphisms of $\text{Mat}_{> 0}(k+m,n)$; see \cite[Theorem 6.7, Corollary 6.8]{Muller_2017}. In particular, $\tau$ maps a positive matrix to a positive matrix.

For $m=4$ the twist map is related to the notion of \emph{parity duality} in physics \cite{Arkani_Hamed_2018}, which has been generalized purely at the level of the amplituhedron for any $m$ \cite{Galashin_2020}.

\section{Convexity and intersections with polytopes}\label{sec:Convexity and intersections with polytopes}

The goal of this section is to define a notion of convexity for semialgebraic sets in real embedded projective varieties, and in particular for semialgebraic subsets of the Grassmannian.  The following definition of semialgebraic set is used in \cite{Arkani_Hamed_2017}, where positive geometries were first defined. 

\begin{definition}[Semialgebraic set]\label{def:semialg}
    A \emph{basic semialgebraic cone} $S$ in $\bR^{N+1}$ is a subset defined by homogeneous equations and inequalities. A \emph{semialgebraic cone} is a finite boolean combination of basic semialgebraic cones. A \emph{semialgebraic set} in real projective space $\bP^N$ is the image of a semialgebraic cone under  $\pi: \bR^{N+1} \setminus \{0\} \to \bP^N.$ Semialgebraic subsets of a projective variety $X$ are semialgebraic subsets of the  ambient projective space which lie in $X.$ 
\end{definition}
We will study convexity in the projective setting. We call a subset $S \subset \mathbb{P}^N$ \emph{convex} if it is of the form $\mathbb{P}(C)$ for some convex set $C \subset \mathbb{R}^{N+1}$, i.e. the image of $C$ under the projection map $\mathbb{R}^{N+1} \setminus \{0\} \rightarrow \mathbb{P}^N, \ x \mapsto [x]$. The notion of convex hull is independent on the choice of affine chart on projective space, or equivalently on the choice of a hyperplane $H$ at infinity, provided that $H$ is disjoint from $S$, see \cite[Lemma 2.2]{Kummer_2022}. 
\begin{definition}\label{def:convexhull}
Let $S \subset \mathbb{P}^N$ be a semialgebraic set.
\begin{enumerate}
    \item The set $S$ is \textit{very compact} if there is a real hyperplane $H \subset \mathbb{P}^N$ such that $H \cap S = \emptyset$.
    \item If $S$ is connected and very compact, then the \textit{convex hull} of $S$ is 
    \begin{equation}
        \conv(S) = \iota_H( \conv(\iota_H^{-1}(S)) \subset \mathbb{P}^N \, ,
    \end{equation}
    where $\iota_H \, \colon \,  \mathbb{R}^N \cong \mathbb{P}^N \setminus H \hookrightarrow \mathbb{P}^N$ is the affine chart associated to $H$.

\end{enumerate}
\end{definition}

\subsection{Extendable Convexity}
\label{subsec:Extendable Convexity}

We now introduce a notion of convexity for semialgebraic sets in a real embedded projective variety $X \subset \mathbb{P}^N$, arising from convexity in the ambient projective space. The main case of interest to us is when $X = \Gr(k,n)$ in its Plücker embedding. This was first considered by Busemann in \cite{busemann1961convexity}.  

\begin{definition}
    Let $X \subset \mathbb{P}^N$ be a real variety and let $S \subset X$ be a connected, semialgebraic, and very compact set. We define the \textit{convex hull of $S$ in $X$} to be
    \begin{equation}
        \conv_X(S) := X \cap \conv(S)  \, .
    \end{equation}
    We say that $S$ is \textit{extendably convex} if $S=\conv_X(S)$.% \elia{Do we want to make part of the definition that $S = \cl_X(X \cap \inter(\conv(S)))$? }
\end{definition}

The convex hull of a semialgebraic set is semialgebraic by the Tarski-Seidenberg Principle; see~\cite[Section 1.4]{bochnak2010real}. Therefore, $\conv_X(S)$ is a semialgebraic set in $X$. Note that $S$ is extendably convex in $X$ if and only if it is equal to the intersection of $X$ with some convex set in $\mathbb{P}^N$.

Because this notion of convexity is inherited from that in projective space, we immediately obtain the following properties, which hold for the convex hull of cones in real vector spaces.

\begin{proposition}
    Let $X \subset \mathbb{P}^N$ be a real variety and let $S,T \subset X$ be connected, semialgebraic, and very compact set. The following properties hold.
\begin{enumerate}
    \item Monotonicity: $S \subset T \implies \conv_X(S) \subset \conv_X(T)$.
    \item Idempotence: $\conv_X(\conv_X(S)) = \conv_X(S)$.
    \item  Antiexchange: If $S$ is extendably convex, $x,y \notin S$, $x\neq y$ and $x \in \conv_X(S \cup \{y\})$, then $y \notin \conv_X(S \cup \{x\})$.
\end{enumerate}
\end{proposition}

If the variety $X$ is invariant under the action of a subgroup of $\PGL(N)$, the automorphism group of $\mathbb{P}^N$, then we obtain an additional equivariance property for the convex hull in $X$. This is the case for the Grassmannian $\Gr(k,n)$, which is a symmetric space given by a quotient of $\GL(n)$. In this case, $\GL(n)$ acts on $\mathbb{P}^{N} \cong \mathbb{P}(\bigwedge^{k}\mathbb{R}^n)$ with $N=\binom{n}{k}-1$ via the exterior product of linear maps.
\begin{proposition}
     Let $X = \Gr(k,n)$ and $S \subset X$ be a connected, semialgebraic, and very compact set. Then, the extendably convex hull is $\GL(n)$-equivariant, namely
    \[\forall \, \sigma \in \GL(n) \,:\, \conv_X(\sigma(S)) = \sigma(\conv_X(S)) \, .\]
\end{proposition}
\begin{proof}
   We combine $\PGL(N)$-equivariance of the convex hull in projective space, where $N= \binom{n}{k}$, with the fact that $\Gr(k,n)$ is stable under the action of $\GL(n)$:
\begin{equation*}
\begin{aligned}
     \hspace{2cm}\sigma(\conv_X(S)) &= \sigma(\conv(S) \cap \Gr(k,n)) = \sigma(\conv_X(S)) \cap \sigma(\Gr(k,n)) \\
    & = \conv(\sigma(S)) \cap \Gr(k,n) = \conv_X(\sigma(S)) \, .\hspace{5.1cm}\qedhere
\end{aligned}
\end{equation*}
\end{proof}
%\begin{remark}
%    There are many possible choices for how to define convexity of subsets of $\Gr(k,n).$ Another natural option would be geodesic convexity, using the Fubini-Study metric inherited from the ambient projective space. Recall that a \emph{geodesic arc} is a locally distance-minimizing path between two points. A subset $S$ of a Riemannian manifold is $X$ is said to be \emph{geodesically convex} if for every two points in $S$ there exists a unique minimizing geodesic arc connecting them and lying entirely in $S.$ With this definition, it turns out that the non-negative Grassmannian itself is not geodesically convex: the geodesic arc between two points in the non-negative Grassmannian may take negative values on Pl\"ucker coordinates. \lizzie{TODO confused about this}
%\end{remark}
Given a subset $S \subset \mathbb{P}^N$, we denote the interior by $\inter(S)$, the closure by $\cl(S)$ and the boundary of $S$ by $\partial S,$ where the boundary is with respect to the Euclidean (quotient) topology on $\mathbb{P}^N$. Similarly, if $S \subset X$ for some subvariety $X \subset \mathbb{P}^N$, we denote the interior by $\inter_X(S)$, the closure by $\cl_X(S)$ and the boundary by $\partial_X S$, all relative to $X$, i.e. with respect to the subspace topology on $X$. It is a standard exercise in topology to show that $\cl_X(S) = X \cap \cl(S)$ and $\inter_X(S) = \inter((X \cap S) \cup X^c).$

\begin{definition}
    Let $S \subset X$ be a semialgebraic set in a real projective variety~$X$. The \textit{algebraic boundary} $\partial_a S$ of $S$ is the Zariski closure of $\partial_X S$ in $X$.
\end{definition}

We first want to establish when the algebraic boundary of a semialgebraic set in a real projective variety is pure of codimension one. Throughout the following, let $X \subset \mathbb{P}^N$ be an irreducible real projective variety. We start from a purely topological definition.

\begin{definition}
    A subset $S \subset X$ is called \textit{regular in $X$} if it is contained in the closure in $X$ of its relative interior in $X$. In formulas, 
    \begin{equation}
         S \subset \cl_X(\inter_X(S)) \, .
    \end{equation}
    We define the \textit{irregular points} of $S$ to be
    \begin{equation}
          \cl_X(\inter_X(S)) \setminus S \, , 
    \end{equation}
    so that $S$ is regular in $X$ if and only if the set of its irregular points is empty.
\end{definition}
For example, if $S \subset X$ is open in $X$, or if it is equal to the closure in $X$ of an open set in $X$, then $S$ is regular in $X$. 
%\begin{lemma}\label{lemma:regularity 2}
%    Let $S \subset X$. Then, $S$ is regular in $X$ if and only if one of the following equivalent conditions holds true.
%    \begin{enumerate}
%        \item For every $x \in S$ and $\epsilon >0$ there exists $y \in S$ and $\delta > 0$ such that $B_{\delta}(y) \cap X \subset B_{\epsilon}(x) \cap S$, where $B_{r}(z)$ denotes the open ball in $\mathbb{P}^N$ centered at $z \in \mathbb{P}^N$ and of radius $r>0$, where we equip $\mathbb{P}^N$ with the metric induced from the 
%        \item $\partial_X S \subset  \partial_X \, \inter_X(S)$.
%    \end{enumerate}
%\end{lemma}

\begin{remark}
    Suppose that $X \subset \mathbb{P}^N$ is non-degenerate, that is, that $X$ is not contained in any projective space of dimension smaller than $N$. Then, any subset $S \subset X$ with non-empty interior affinely spans $\mathbb{P}^N$, and in particular $\conv(S) \subset \mathbb{P}^N$ is full-dimensional.
      %Note that if a set $S \subset X$ has non-empty interior in $X$, then $S$ affinely spans $\mathbb{P}^N$. In particular, $\conv(S) \subset \mathbb{P}^N$ is full-dimensional. Indeed, if $U \subset S$ is any subset open in the Euclidean topology in $X$, then $U$ affinely spans $\mathbb{P}^N$. Indeed, suppose that there exists a hyperplane $H \subset \mathbb{P}^N$ containing $U$. Then $H \cap X$ is a nonempty closed algebraic subset of $X$ containing $U$. Since $U$ is open in $H \cap X$, $X$ must contain the Zariski closure of $H \cap X$, i.e. $X$ contains $H$. This contradicts $X$ being non-degenerate.
\end{remark}

\begin{eg}\label{example:half-sphere with handle}
In projective space, each convex semialgebraic set $S$ with non-empty interior is regular, and so is its complement \cite[Remark 2.4]{sinn2014algebraicboundariesconvexsemialgebraic}. This does not hold in general for extendably convex sets in real projective varieties. % If $S \subset \mathbb{P}^N$ is regular, then $X \cap S$ is not necessarily regular in $X$. 
As an example, we work in an affine chart of $\mathbb{P}^3$ and take $X = \mathcal{V}(x^2+y^2+z^2-1) \subset \mathbb{R}^3$ to be the sphere of unit radius, and $S$ to be the convex hull in $\mathbb{R}^3$ of the union of the closed upper hemisphere of $X$ with a great arc segment in the lower hemisphere,
    \begin{equation}
        S = \conv(X \cap \{(x,y,z) : z \geq 0\} \cup \{X \cap \mathcal{V}(y,z)\}\}) \, .
    \end{equation}
    Then, $S$ is convex, and therefore $X \cap S$ is convex in $X$. However, $X \cap S$ is not regular in $X$, as the closure of the interior does not contain the arc segment. On the other hand, since $S$ is closed in $X$, the complement $X \setminus S$ is open and hence regular in $X$.
\end{eg}

The following is an extension of \cite[Lemma 2.5]{Ranestad_2024}. The same proof can be applied to this setting since the argument is purely local. 

\begin{lemma}\label{lemma:alg boundary pure of cod 1}
    Let $\emptyset \neq S \subset X$ be a semialgebraic set, regular in $X$, and suppose that $X \setminus S$ is non-empty and regular in $X$. Then each irreducible component of the algebraic boundary of $S$ has codimension equal to one in $X$, i.e. $ \partial_a S$ is a hypersurface in $X$.
\end{lemma}

\subsection{Intersections with convex polytopes }
\label{subsec:Intersections with convex polytopes}

In this section we give a criterion for when a semialgebraic set $S$ in a real projective variety $X$ is equal to the intersection of $X$ with a convex set $P$ in the ambient projective space. Our main case of interest is when $X$ is the Grassmannian and $P$ is a projective polytope. The results presented here will be crucial in proving Theorem~\ref{thm:int_k=m=2}. We start with some elementary results from topology.

\begin{lemma}\label{lemma:U=V}
    Let $X$ be a topological space and $U,V \subset X$ open subsets such that $U$ is connected, $\emptyset \neq V \subset U$ and $\partial V \subset \partial U$. Then $V=U$.
\end{lemma}

\begin{proof}
    Since $U$ is open, $U \cap \partial U = \emptyset$, and therefore by assumption also $U \cap \partial V = \emptyset$. Then
    \begin{equation}
        U = V \cup (U \setminus \overline{V}) \cup (U \cap \partial V) = V \cup (U \setminus \overline{V}) \, .
    \end{equation}
    Since $U$ is connected and $V$ is nonempty, it follows that $U \setminus \overline{V}= \emptyset$. Hence $U = V$.
\end{proof}

Let us now refine this result for our setting of interest.

\begin{lemma}[Intersection with polytope]\label{lemma:inters with polytope}
    Let $S$ be a regular closed semialgebraic set in a real irreducible smooth non-degenerate projective variety $X \subset \mathbb{P}^N$. Let $P$ be a full-dimensional projective convex polytope in $\mathbb{P}^N$. Assume that the following conditions hold, where all inclusions are in $\bP^N$:
    \begin{enumerate}
        %\item $\partial_a S$ is a union of linear divisors in $X$, i.e. divisors of the form $X \cap H $ with a hyperplane $H \subset \mathbb{P}^N$.
        \item $X \cap P$ is regular in $X$,
        \item $X \cap {\rm int}(P)$ is connected,
        
        %\item ${\rm vert}(S) \subset X \cap \Ex(P)$, %$\partial_a S \subset \partial_a P$, or equivalently, ${\rm vert}(S) = X \cap \Ex(P)$.
        \item  %$\inter_X(S)$ is connected and 
        $\emptyset \neq \inter_X(S) \subset \inter( P)$,
        \item $\partial_a S \subset \partial_a P$.
        
    \end{enumerate}
    Then 
    \begin{equation}
        S = X \cap P \, , \quad \text{and} \quad \partial_X S = X \cap \partial P  \, .
    \end{equation}
   In particular, $S$ is extendably convex. 
\end{lemma}
\begin{proof}

We first show that $\partial_X(X \cap P) = X \cap \partial P.$ It is a simple exercise to show that $\partial_X(X \cap P) \subset X \cap \partial P,$ so it suffices to show the reverse inclusion. We argue by contradiction. Let $p$ be a point in $X \cap \partial P,$ and suppose that $p$ is not in $\partial_X(X\cap P).$ Then $p$ is in the interior of $X \cap P,$ and $X \cap \partial P$ has the subspace topology in $X \cap P,$ so there is an open neighborhood of $p$ in $X$ contained in $X \cap \partial P$. Since $X$ is irreducible, it is equal to the Zariski closure of any of its open subsets. Thus   $X$ must be contained in $\partial_a P$. Since $P$ is a polytope, $\partial_a P$ is union of hyperplanes in $\mathbb{P}^N$. But $X$ is non-degenerate, so it cannot be contained in a hyperplane.

We now prove the inclusion $\partial_X S \subset \partial P$. This follows from $\partial_a P \cap P = \partial P$, as $P$ is a projective convex polytope. In fact, by assumption~3 and the fact that $S$ is regular we have that $\partial_X S = \partial_X \inter_X (S) \subset P$. On the other hand, by assumption~4 we have that $\partial_X S \subset \partial_a S \subset \partial_a P$, and intersecting with $P$ we obtain that $\partial_X S \subset \partial_a P \cap P = \partial P$. 

We now prove the result. Let $V = \inter_X(S)$ and $U=\inter(X \cap P)$, which are open subsets of~$X$. Since $S$ is regular in $X$, $\partial_X S = \partial_X V$. Similarly, by assumption~1 it follows that $\partial_X(X \cap P) = \partial_X U$. Then, by assumption~3 we have that $\partial V \subset X \cap \partial P = \partial_X U \subset \partial U$. Therefore, we can apply Lemma~\ref{lemma:U=V} to deduce that $U=V$. In particular, $\partial U = \partial V $. Moreover, by the fact that $S$ is regular, together with assumption~1, we deduce that $S = \cl_X(V) = \cl_X(U) = X \cap P $.
%$\textit{(c)}$ %By assumption 2, $\partial S \subset \partial_a S \subset \partial_a P$, and by assumption 1 it follows that $\partial S \subset \partial_a P \cap P = \partial P$, where the last equality holds since $P$ is a convex polytope. Again from assumption 1, $\partial P$ does not intersect the relative interior of $S$, so that $X \cap \partial P \subset $ 
%This follows from
%\begin{equation}
%    X \cap \partial P = \partial_X (X \cap P) = \partial_X S \, ,
%\end{equation}
%where the first equality follows from assumption 2, and the second one is $\textit{(a)}$.
\end{proof}
We now characterize a subset of the potentially irregular points in the intersection of a projective variety with a polytope. This will be useful later when checking the assumptions of Lemma~\ref{lemma:inters with polytope}.

%\begin{proposition}
%Let $X \subset \mathbb{P}^N$ be a real projective variety and let $P \subset \mathbb{P}^N$ be a full-dimensional projective convex polytope in $\mathbb{P}^N$ such that $X \cap \inter(P) \neq \emptyset$. Assume that $X \cap \partial P$ is connected, then $X \cap P$ is regular.
%\end{proposition}

%\begin{proof}
%Let $E :=  (X \cap \partial P) \setminus \partial_X(X \cap \inter(P))$. Note that if $E = \emptyset$, then $X \cap P$ is regular, hence assume that $E \neq \emptyset$. Let $x \in E$. Since $X \cap \partial P$ is connected, $\cl_X(E) \cap \partial_X(X \cap \inter(P)) \neq \emptyset$, so let $y$ be an element in this set.
%Again, since $X \cap \partial P$ is connected, there exists a face $F \subset X \cap \partial P$ of $P$ containing both $x$ and $y$. Let $\overline{F}$ denote the real points of the Zarisky closure of $F$. Since $X$ and $P$ are real and $\overline{F}$ is linear, $\overline{F} \subset X$. We claim that 
%\end{proof}

\begin{lemma}\label{lemma:reg_X_P}
Let $X \subset \mathbb{P}^N$ be a real projective variety of codimension $r$ and let $P \subset \mathbb{P}^N$ be a full-dimensional projective convex polytope in $\mathbb{P}^N$ such that $X \cap \inter(P) \neq \emptyset$. Let $F \subset \partial P$ be a face of $P$ of dimension greater or equal than~$r$. Then, the irregular points of $X \cap P$ in the relative interior of $F$ in $X$ are contained in the singular locus of $X \cap F$, i.e.
\begin{equation}
    (\cl_X( \inter_X(X \cap P)) \setminus (X \cap P)) \cap {\rm relint}_X(F) \subset {\rm Sing}(X \cap \overline{F}) \cap {\rm relint}_X(F) \, ,
\end{equation}
where $\overline{F}$ denotes the Zariski closure of $F$ in $\mathbb{P}^N$ and ${\rm relint}_X(F) := {\rm int}_{\overline{F}}(X \cap F)$.
\end{lemma}

\begin{proof}
    Let $p \in X \cap {\rm relint}_X(F) $. Note that $p$ is a regular point in $X \cap F$ if and only if $X$ and $F$ intersect transversally at $p$. By the assumption on the dimensions, this is equivalent to $\dim(T_pF) + \dim(T_pX) = N$. In turn, since $p \in {\rm relint}_X(F)$ and by dimension reasons, this is equivalent to any small open ball in $X$ around $p$ intersecting $\inter(P)$, that is, to $p$ being regular.
\end{proof}

This criterion is particularly useful if $X$ is a hypersurface, and if one understand the singular locus of the intersection of $X$ with each face of $P$. This will be the case for us in Section~\ref{sec:Convexity of Amplituhedra}.

\section{Exterior cyclic polytopes} \label{sec:Exterior Cyclic Polytopes}
In this section we define and study the exterior cyclic polytope associated to a positive matrix. These polytopes are of independent interest for combinatorics, beyond applications to amplituhedra. We study the facet structure, $f$-vector, and dependence on the input matrix.

\begin{definition}
    Let $P \subset \bP^{r}$ be a projective polytope. We define the $k$-th exterior polytope  $\bigwedge^kP \subset \bP(\bigwedge^k \bR^{r+1})$ of $P$ to be the convex hull of all points $v_1 \wedge \dots \wedge v_k$ such that all $v_i$ are vertices of $P$.
\end{definition}

Here the convex hull is taken in projective space, as in Definition \ref{def:convexhull}. Since the set of vertices of $P$ is finite, it is very compact, and so the convex hull is independent of choice of affine chart.

\begin{definition}
    Fix $Z \in \text{Mat}_{> 0}(k+m,n)$ with $n \geq k+m$. Let $C_{k+m,n}(Z)$ be the cyclic polytope given by the convex hull of the columns of $Z$. We define the \textit{$k$-th exterior cyclic polytope $C_{k,m,n}(Z)$ of type} $(m,n)$ to be $ \bigwedge^k C_{k+m,n}(Z)$.
\end{definition}

We use the indexing $k+m$ because the exterior polytope is empty for $m<0$, and also to match the amplituhedron convention. We note the following.
\begin{lemma}
    The polytope $ C_{k,m,n}(Z)$ is the image in $\mathbb{P}(\bigwedge^k \mathbb{R}^{k+m})$ of the standard simplex in $\bP(\bigwedge^k\bR^n)$ under the linear map $\bigwedge^k Z$. 
\end{lemma}

In particular, $C_{k,m,n}(Z)$ is a projective polytope. It is the convex hull of 
\begin{equation}\label{vertices}
    (i_1 \dots i_k) := Z_{i_1} \wedge \dots \wedge Z_{i_k} \, , \qquad 1 \leq i_1 < \dots < i_k \leq n \, .
\end{equation}
Each such vertex is the projection of $e_I := e_{i_1} \wedge \ldots \wedge e_{i_k}.$ The facets of the exterior cyclic polytope are certain hyperplanes in $\bigwedge^k \bR^{k+m}$ spanned by a subset of the vertices. They are certain flats of rank $\binom{k+m}{2}-1$ in the linear matroid of the matrix $\wedge^k Z.$ These sub-maximal flats are called \emph{hyperplanes} of the matroid. It will often be convenient to study this matroid.

\begin{definition}\label{def:wedge_matroid}
    The \emph{wedge power matroid} $W_{k,m,n}$ is the linear matroid of $\wedge^k Z$ for $Z$ a generic $(k+m) \times n$ matrix.
\end{definition}

\begin{eg}
    The columns of $\wedge^kZ$ will have dependencies. For example, let $Z_1, \, \ldots, \, Z_6$ be the columns of a $4 \times 6$ matrix. Then we may write $Z_1= aZ_2+ bZ_3+cZ_4+ dZ_5$ for some scalars $a,b,c,d.$
    Taking wedge product with $Z_1$ on both sides produces a linear relation between the columns ${Z_1 \wedge Z_2}, \, Z_1 \wedge Z_3, \, Z_1 \wedge Z_4,$ and $Z_1 \wedge Z_5$ of $\wedge^2Z.$
\end{eg}

We note that starting from $m=2,$ the linear matroid of $\wedge^kZ$ depends not only on the matroid of $Z$ but also on the matrix $Z$ itself, as we will see in Subsection \ref{subsec:The case $m=k=2$ and $n=6$}. However, the matroid of $\wedge^kZ$ is still constant for a generic choice of $Z$; indeed, the vanishing of minors of $\wedge^kZ$ defines a Zariski closed subset of the set of $(k+m) \times n$ matrices. Thus Definition \ref{def:wedge_matroid} makes sense. 

We will see in Subsection \ref{subsec:The case $m=k=2$ and $n=6$} that the combinatorial type of $C_{k,m,n}(Z)$ varies as $Z$ varies over positive matrices. However, we may still prove some independence of $Z \in \text{Mat}_{> 0}(k+m,n)$.

\begin{lemma}\label{lem:slinvariant}
The oriented matroid of $\wedge^kZ$ , and thus the combinatorial type of the exterior cyclic polytope, depends only on the image of $Z$ in $\Gr_{> 0}(k+m,n).$
\end{lemma}
\begin{proof}
    The oriented matroid of $\wedge^kZ$ is given by the sign of the maximal minors of $\wedge^kZ.$ Each such minor is a polynomial in the matrix entries of $Z.$ We claim each such polynomial is invariant under the ${\rm SL}(k+m)$ action on $Z$. Indeed, any matrix $A \in {\rm SL}(k+m)$ induces an endomorphism $\wedge^kA$ of $\bigwedge^k\bR^{k+m}$ by acting on each tensor factor. The determinant of $\wedge^kA$ is a power of $\det A,$ so for any $\binom{k+m}{k} \times \binom{k+m}{k}$ submatrix $W,$ we have $\det (\wedge^kA \cdot W) = \det(W).$ 
\end{proof}

The relationship between the oriented matroid of a set of points and the convex hull of said points is described in \cite[Remark 3.7.3]{berndmatroids}. Oriented matroids have many cryptomorphic axiomatizations, including one in terms of \emph{covectors}. The covectors of the oriented matroid coming from a point configuration are signed sets $Y = (Y^+, Y^-)$ such that there exists a hyperplane $H$ with  $Y^+ = Y \cap H^+$ and $Y^- = Y \cap H^-.$ Thus a facet of the convex hull of these points is exactly the span of a set of points $Y$ such that the signed set $(Y , \varnothing)$ is a covector. This shows that the oriented matroid of $\wedge^kZ$ determines the combinatorial type of the exterior cyclic polytope $C_{k,m,n}(Z)$.

\subsection{The case $m=0$} 
\label{subsec:The case $m=0$}

Here $\wedge^kZ$ is a $1 \times \binom{n}{k}$ matrix whose entries are the maximal minors of $Z.$ For a general matrix, these minors are nonzero; thus the matroid $W_{k,0,n}(Z)$ is equal to the uniform matroid $U(1, \binom{n}{k})$. The exterior cyclic polytope in this case is equal to the single point $\mathbb{P}^{0}$.

\subsection{The case $m=1$}
\label{subsec:The case $m=1$}

The case $m=1$ is already interesting. In this case, the exterior cyclic polytope lives in $\bP(\bigwedge^{k}\bR^{k+1}) = (\bP^k)^\vee.$ The columns of $Z$ can be viewed as a configuration of points in $\bP^k$ or in $\bR^{k+1}.$
\begin{definition}
    Let $M$ be a point configuration in $\bR^k.$ The \emph{discriminantal arrangement} of $M$ is the hyperplane arrangement consisting of all hyperplanes spanned by points in $M.$ 
\end{definition}
\begin{proposition}
   The matroid $W_{k,1,n}$ is the linear matroid of the discriminantal arrangement of $n$ general points in $\bR^{k+1}.$ 
\end{proposition}
\begin{proof}
    Each column $Z_I = Z_{i_1} \wedge \ldots \wedge Z_{i_k}$ in the $(k+1) \times n$ matrix $\wedge^kZ$ is the normal vector to a hyperplane in $\bR^{k+1}$ passing through $k$ of the $n$ points. The matroid of a hyperplane arrangement is exactly defined as the linear matroid of the normal vectors.
\end{proof}
\begin{figure}[!h]
\begin{center}
    \scalebox{0.6}{\begin{tikzpicture}[scale=3]
    % Define the 5 points and place labels outward
    \foreach \i in {1,...,5} {
        \pgfmathsetmacro\angle{72*(\i-1)}
        \pgfmathsetmacro\x{cos(\angle)}
        \pgfmathsetmacro\y{sin(\angle)}
        \coordinate (P\i) at (\x,\y);
        \fill (P\i) circle (0.02);
        % Label pushed outward along radial direction
        \node[font=\small] at ({1.15*cos(\angle)}, {1.15*sin(\angle)}) {\i};
    }

    % Draw extended lines for all 10 pairs
    \foreach \i/\j in {1/2, 1/3, 1/4, 1/5, 2/3, 2/4, 2/5, 3/4, 3/5, 4/5} {
        \path let \p1 = (P\i), \p2 = (P\j) in
            coordinate (A) at ($(\p1) - 0.2*(\p2) + 0.2*(\p1)$)
            coordinate (B) at ($(\p2) - 0.2*(\p1) + 0.2*(\p2)$);
        \draw[blue, thin] (A) -- (B);
    }

\end{tikzpicture}}
\end{center}
\caption{The discriminantal arrangement of five general points in $\bP^2$}\label{fig:disc5}
\end{figure}
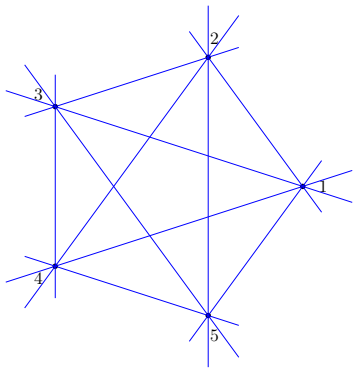
For example, the matroid $W_{k,1,k+2}$ is the graphic matroid of the complete graph $K_{k+2}.$ Equivalently, it is the discriminantal arrangement of $k+2$ general points in $\bR^{k+1}.$ This hyperplane arrangement is combinatorially equivalent to the braid arrangement. % Indeed, if one chooses the $k+2$ points to be $e_1, \ldots, e_{k+1}, e_1 + \ldots + e_{k+1},$ one obtains the projection of the braid arrangement in $\bR^{k+2}$ to $\bR^{k+2}/\bR\gen{1, \, \ldots, \, 1}.$

For $m=1$, the exterior cyclic polytope $C_{k,1,n} \subset (\mathbb{P}^{k})^\vee$ is a $k$-simplex. The amplituhedron, on the other hand, is the union of bounded regions of the hyperplane arrangement whose normals are the columns of $\wedge^kZ$ \cite{Karp_2017}. For $k=3$ and $n=5$ the exterior cyclic polytope and the amplituhedron live in $(\bP^2)^\vee$ and are depicted in Figure~\ref{fig:m=1amp}. Notice that projective duality exchanges the five points in Figure \ref{fig:disc5} with the five lines in Figure~\ref{fig:m=1amp}.
\begin{figure}[!h]
\begin{center}
    \scalebox{0.6}{\tikzset{every picture/.style={line width=0.75pt}} %set default line width to 0.75pt        

\begin{tikzpicture}[x=0.75pt,y=0.75pt,yscale=-1,xscale=1]

%Shape: Right Triangle [id:dp1927140282881451] 
\draw  [fill={rgb, 255:red, 184; green, 233; blue, 134 }  ,fill opacity=1 ] (140,149.5) -- (239.75,200) -- (140,200) -- cycle ;
%Shape: Right Triangle [id:dp285589168793292] 
\draw  [fill={rgb, 255:red, 184; green, 233; blue, 134 }  ,fill opacity=1 ] (140,120) -- (220,200) -- (140,200) -- cycle ;
%Shape: Right Triangle [id:dp8643070912773717] 
\draw  [fill={rgb, 255:red, 184; green, 233; blue, 134 }  ,fill opacity=1 ] (140,80) -- (200,200) -- (140,200) -- cycle ;
%Straight Lines [id:da7970580763235176] 
\draw    (140,40) -- (140,240) ;
%Straight Lines [id:da674532767475405] 
\draw    (80,200) -- (320,200) ;
%Straight Lines [id:da5408061349136138] 
\draw    (120,40) -- (220,240) ;
%Straight Lines [id:da9705286776727455] 
\draw    (100,80) -- (260,240) ;
%Straight Lines [id:da4861392665580626] 
\draw    (80,120) -- (280,220) ;
%Shape: Right Triangle [id:dp3219181418147359] 
\draw  [fill={rgb, 255:red, 74; green, 144; blue, 226 }  ,fill opacity=0.3 ] (440,80) -- (540,200) -- (440,200) -- cycle ;
%Straight Lines [id:da5709279430079148] 
\draw    (440,40) -- (440,240) ;
%Straight Lines [id:da9747828774688085] 
\draw    (380,200) -- (620,200) ;
%Straight Lines [id:da7284313277119369] 
\draw    (420,40) -- (520,240) ;
%Straight Lines [id:da30395124701786] 
\draw    (400,80) -- (560,240) ;
%Straight Lines [id:da20732610970450216] 
\draw    (380,120) -- (580,220) ;

% Text Node
\draw (141,22) node [anchor=north west][inner sep=0.75pt]   [align=left] {$\displaystyle 1$};
% Text Node
\draw (101,22) node [anchor=north west][inner sep=0.75pt]   [align=left] {$\displaystyle 2$};
% Text Node
\draw (81,60) node [anchor=north west][inner sep=0.75pt]   [align=left] {$\displaystyle 3$};
% Text Node
\draw (68,120) node [anchor=north west][inner sep=0.75pt]   [align=left] {$\displaystyle 4$};
% Text Node
\draw (61,200) node [anchor=north west][inner sep=0.75pt]   [align=left] {$\displaystyle 5$};
% Text Node
\draw (448,22) node [anchor=north west][inner sep=0.75pt]   [align=left] {$\displaystyle 1$};
% Text Node
\draw (408,22) node [anchor=north west][inner sep=0.75pt]   [align=left] {$\displaystyle 2$};
% Text Node
\draw (381,62) node [anchor=north west][inner sep=0.75pt]   [align=left] {$\displaystyle 3$};
% Text Node
\draw (361,122) node [anchor=north west][inner sep=0.75pt]   [align=left] {$\displaystyle 4$};
% Text Node
\draw (368,202) node [anchor=north west][inner sep=0.75pt]   [align=left] {$\displaystyle 5$};

\end{tikzpicture}}
\end{center}
    \caption{Amplituhedron (left) and exterior cyclic polytope (right) in $(\bP^2)^\vee$ for ${k=3,m=1,n=5 }$.}\label{fig:m=1amp}
\end{figure}
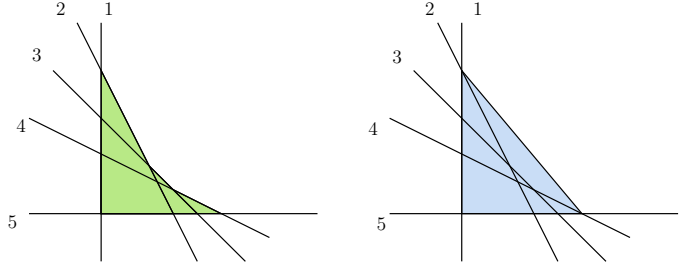

\subsection{The case $m=k=2$ and $n=6$}
\label{subsec:The case $m=k=2$ and $n=6$}
We present an extended example. Fix real numbers $a<b<c<d<e<f$. This gives a positive~matrix
\begin{equation}
\label{eq:Zmatrix} Z \quad = \quad \begin{pmatrix}
1 & 1 & 1 & 1 & 1 & 1  \\
a & b & c & d & e & f  \\
a^2 & b^2 & c^2 & d^2 & e^2 & f^2  \\
a^3 & b^3 & c^3 & d^3 & e^3 & f^3  
\end{pmatrix} \, .
\end{equation}
The convex hull of the columns of $Z$ is the cyclic polytope $C_{4,6}$,
which is the convex hull of the union of two triangles that meet in a unique point in the interiors. Its $f$-vector is $(6,15,18,9)$. Then $\wedge^2Z$ is given by the $6 \times 15$ matrix whose entries are the $2 \times 2 $ minors:
$$ \wedge^2 Z \,\, = \,\,
\begin{pmatrix}
a-b &  a-c & a-d & a-e & \cdots &   d-f & e-f \\
a^2-b^2 & a^2-c^2 & a^2-d^2 & a^2-e^2 & \cdots &     d^2-f^2 &  e^2-f^2 \\
a^3-b^3 & a^3-c^3 & a^3-d^3 & a^3-e^3 & \cdots &   d^3-f^3 &  e^3-f^3 \\
a^2 b-a b^2 &  a^2c-ac^2 & a^2 d-a d^2 & a^2e-ae^2 & \cdots &  d^2 f-d f^2 & e^2f-ef^2 \\
a^3 b-a b^3 &  a^3c-ac^3 & a^3 d-a d^3 & a^3e-ae^3 & \cdots & d^3 f-df^3 & e^3 f-e f^3 \\
a^3 b^2-a^2 b^3 & a^3c^2-a^2c^3 & a^3d^2-a^2d^3 & a^3e^2-a^2e^3 & \cdots & \! d^3 f^2-d^2 f^3 & e^3f^2-e^2f^3 \\
\end{pmatrix} \,. 
$$
The (cone over the) exterior cyclic polytope is the convex hull of the columns of $\wedge^2 Z$ in $\bigwedge^2 \bR^4.$ We identify $\wedge^2 Z$ with the matrix that is obtained by dividing the columns by their common factors, $a-b,a-c,a-d, a-e,\ldots, d-f,e-f $. The resulting matrix has $\binom{16}{5} = 5005$ maximal minors, each of which is a homogeneous polynomial of degree $12$. Of these minors, $1660$ are zero and the other $3345$ are nonzero. These bases come in $12$ symmetry classes. We identify the ground set $\binom{[6]}{2}$ with edges of a complete graph on $6$ vertices, and depict these bases in Figure \ref{fig:c246bases}. 
\begin{figure}[!h]
\begin{center}
    \includegraphics[scale=0.2]{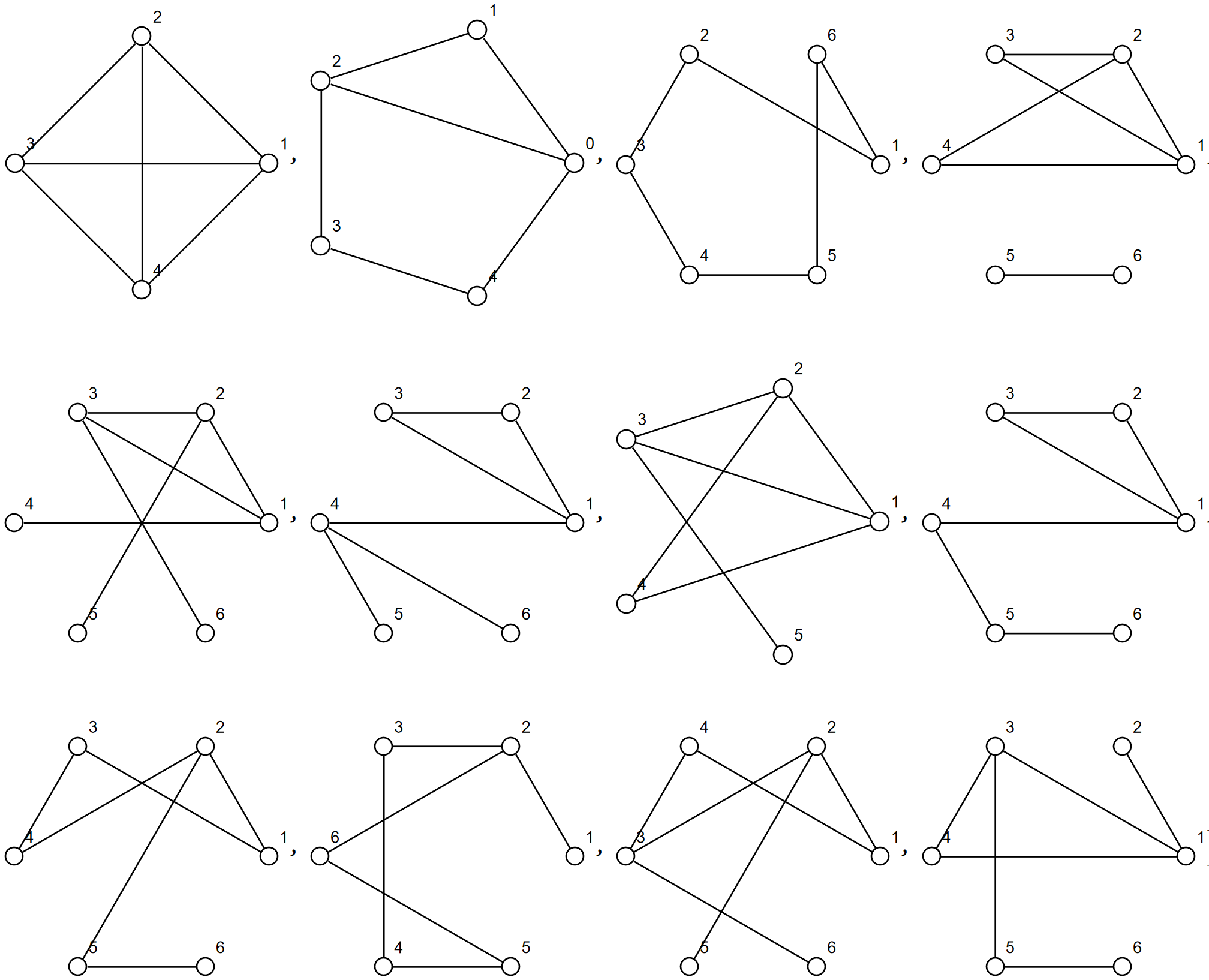}
    \caption{Bases of $W_{2,2,6}$}\label{fig:c246bases}
    \end{center}
\end{figure}
In all classes but one, the sign of the minor
 is fixed by the ordering of $a,b,c,d,e,f$.
The only exception is the {\em cycle}:
\begin{align*}
 [16,12,23,34,45,56] = (a{-}c)(a{-}d)(a{-}e)(b{-}d)(b{-}e)(b{-}f)(d{-}f)   (c-e)(c-f)  \\
{ } \qquad \hfill
   \cdot (abd-abe-acd+acf+ade-adf+bce-bcf-bde+bef+cdf-cef) \, .   
\end{align*}
Here the minor has
a cubic factor with $12$ terms which can have any sign. For instance, fix $( a,b,c,d,e) = (1,3,4,7,8).$ Then the cubic seen above is positive for $f > 47/5$, zero for $f=47/5$, and negative for $f < 47/5$. This shows that both the matroid and the oriented matroid of $\wedge^2 Z$ may change as $Z$ varies over positive matrices.

The combinatorics of $C_{k,m,n}(Z)$ also changes. For $f = 9$, our $5$-dimensional polytope has $f$-vector $(15,75,143,111,30)$. Among the $30$ facets, there are $18$ simplices and six double pyramids over the pentagons, like $\{12,13,14,15,16,23,56\}$. 
The $18$ simplex facets include
$$
\{ 12, 23, 34, 45, 56 \},\,\, \{ 12, 23, 34, 56, 16 \},\,\, \{ 12, 16, 34, 45, 56 \} \, .$$
For $f = 47/5$, these three simplices lie in a hyperplane, and they are replaced by one facet $$  \{ 12, 23, 34, 45, 56, 16 \} \, ,$$
which is a cyclic polytope $C_{4,6}$.
And, for $f > 47/5$, this is replaced by  three other simplex facets:
$$ \{12, 16, 23, 34, 45\}, \,\,   \{ 12, 16, 23, 45, 56 \} \,\, \{ 16, 23, 34, 45, 56 \} \, .$$
The $f$-vector is the same for $f < 47/5$ and $f > 47/5.$
\subsection{The case $m=k=2$}
\label{subsec:The case $m=k=2$}
When $k=2,$ the ground set of the matroid $W_{2,m,n}$ may be identified with the edges of a complete graph on $n$ labeled vertices. This matroid has been studied in the contexts of graph connectivity, graph rigidity, and algebraic statistics \cite{brakensiek2024rigidity, Crespo_Ruiz2023}. In particular, setting $d = n-m-2,$ the matroid of $\wedge^2Z$ is equal to the dual of the \emph{hyperconnectivity matroid} $\cH_{d}(n)$ by \cite[Theorem 1.1]{brakensiek2024rigidity}. This matroid was introduced by Kalai to understand highly connected graphs \cite{kalai}, and has rank $d \, n - \binom{d+1}{2}$ by Property~1 of \cite{kalai}. Thus the rank of its dual is indeed $\binom{n}{2} - d \, n +\binom{d+1}{2}=  \binom{n-d}{2} = \binom{k+m}{2}.$ The hyperconnectivity matroid $\cH_d(n)$ is also equal to the algebraic matroid of $n \times n$ skew-symmetric matrices of rank at most $d$ \cite[Proposition 3.1]{Crespo_Ruiz2023}. A characterization of the graphs which represent independent sets in $\cH_d(n)$ is known only for $d=2$ \cite{bernstein}.

To simplify combinatorial analyses in the hyperconnectivity matroid, we introduce the operations of cutting and gluing on graphs. 

\begin{definition}[Cutting and gluing]\label{lem:gluing}
    Let $G$ be a simple undirected graph on $n$ vertices.
    \begin{enumerate}
        \item Let $e=uv$ be an edge in $G.$ The $\text{Cut}(G,e,v)$ is the graph on $n+1$ vertices obtained by introducing a new vertex $v'$ and replacing $uv$ with $uv'.$
        \item Let $u$ and $v$ be vertices of $G$ with distance at least three. Then $\text{Glue}(G,u,v)$ is the graph on $n-1$ vertices obtained by identifying $u$ and $v$ in the list of edges.
    \end{enumerate}
\end{definition}

The distance condition for gluing guarantees that the new graph will be simple. Both operations preserve the number of edges. Furthermore, they are inverses: indeed, $\text{Glue}(\text{Cut}(G,e,v),v,v') = G.$ We observe the following additional property.
\begin{lemma}\label{lem:cutglue1}
    Let $G$ be a simple undirected graph on $n$ vertices. 
    \begin{enumerate}
        \item  If $G$ is an independent set in $W_{2,m,n}$ then any cutting of $G$ is an independent set in $W_{2,m,n+1}.$
        \item If $G$ is a dependent set in $W_{2,m,n}$ then any gluing of $G$ is a dependent set in $W_{2,m,n-1}.$
    \end{enumerate}
    
\end{lemma}
For example, all circuits in $W_{2,2,n}$ may be obtained by gluings of the three circuits in Figure \ref{fig:circuits4}. Lemma \ref{lem:cutglue1} allows us to display the circuit data for all $n$ in a compact form.
\begin{figure}[!h]
    \begin{center}
    \includegraphics[width=0.6\linewidth]{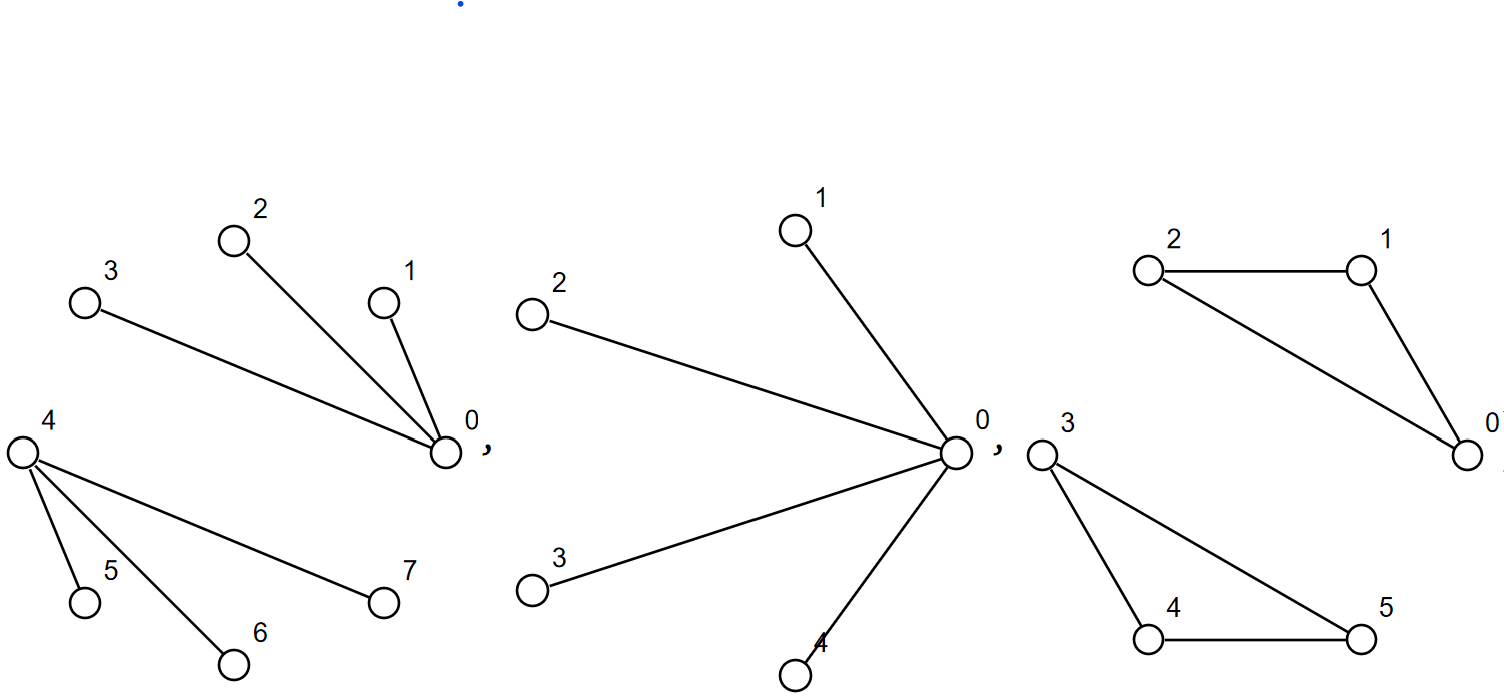}
    \end{center}
    \caption{Circuits in $W_{2,2,n}$ up to gluing}\label{fig:circuits4}
\end{figure}

The $f$-vector of the exterior cyclic polytope is given in Table \ref{table:fvector} for $k=m=2$ and $n$ small. We checked computationally with many positive $Z$ by taking a random Vandermonde matrix, computing its $LU$ decomposition, normalizing $U,$ and setting $Z = U^TDL$ where $D$ is the diagonal of the first $k \times k$ minor of $U.$ Beginning with $n=7,$ Vandermonde matrices themselves (before reversing the $LU$ decomposition) do not give a random enough sample; all exterior cyclic polytopes of Vandermonde matrices which we computed have $72$ facets for $n=7$, rather than the general facet count of $82.$ The data in \ref{table:fvector} should be regarded as a conjecture that even though the combinatorial type changes as $Z$ varies, the $f$-vector remains constant. 
\begin{table}[h!]
\centering
\renewcommand{\arraystretch}{1.2}
\begin{tabular}{rrrrrrr}
$n = 5\ :$  & 10 & 35  & 55   & 40  & 12  & 1 \\
$n = 6\ :$  & 15 & 75  & 143  & 111 & 30  & 1 \\
$n = 7\ :$  & 21 & 147 & 328  & 282 & 82  & 1 \\
$n = 8\ :$  & 28 & 266 & 664  & 616 & 192 & 1 \\
$n = 9\ :$  & 36 & 450 & 1217 & 1191& 390 & 1 \\
\end{tabular}
\caption{The $f$-vectors for $n = 5$ through $n = 9$}\label{table:fvector}
\end{table}

The rest of this subsection will be dedicated to the proof of the following theorem. Let $S_n$ act on the matrix $Z$ by permuting the columns.
\begin{theorem}\label{thm:combotype}
    The matroid of $\wedge^2Z$ for $Z \in {\rm Mat}_{>0}(4,n)$ is equal to the wedge power matroid $W_{2,2,n}$ outside of the closed locus where the polynomial $\det [Z_{12} \ Z_{23} \ Z_{34} \ Z_{45} \ Z_{56} \ Z_{16}] $ or one of its permutations is zero.
\end{theorem}

In particular, this shows that the combinatorial type of $C_{2,2,n}(Z)$ is constant on each region in the complement of the hypersurfaces described in Theorem \ref{thm:combotype}. The polynomial in the theorem may also be written in Pl\"ucker coordinates of $Z$ as 
\begin{equation}
    p_{1234}p_{1356}p_{2456}-p_{1235}p_{1346}p_{2456}+p_{1235}p_{1246}p_{3456} \, .
\end{equation}
If one instead takes maximal minors $q_{ij}$ of any matrix whose rowspan is the kernel of $Z$ and applies Pl\"ucker relations, this polynomial obtains the simpler form $q_{12}q_{34}q_{56}-q_{23}q_{45}q_{16}.$ This expression appears in physics in the algebraic prefactor of the six-dimensional scalar hexagon integral, see~\cite[Section 2.5, Equation 42]{hexagon}. 

The proof of Theorem \ref{thm:combotype} is a computational check, which we present as follows. The question we must answer is: for which bases $\{i_1j_1, \, \ldots, \, i_6j_6\}$ in $W_{2,2,n}$ does there exist a positive matrix $Z$ with $Z_{i_1} \wedge Z_{j_1}, \, \ldots, \, Z_{i_6} \wedge Z_{j_6}$ dependent? We call a basis where such $Z$ exists \emph{dynamic}, and a basis without this property \emph{static}. For example, we saw in Subsection \ref{subsec:The case $m=k=2$ and $n=6$} that the basis $\{12,23,34,45,56,16\}$ is dynamic. 

\begin{lemma}\label{lem:staticcut}
    Suppose that the basis represented by a graph $G$ is static. Then any basis obtained by cutting $G$ is static.
\end{lemma}
\begin{proof}
    We prove the contrapositive. Suppose that there exists a positive matrix $Z$ for which the edges of $\text{Cut}(G,e,v)$ represent a set of dependent columns in $\wedge^kZ.$ This means we have some $Z_{iv}, Z_{jv'}$ among these dependent columns. If $v = v',$ then the columns remain dependent.
\end{proof}
These operations give a poset structure to the $47$ combinatorial types of bases of $W_{2,2,n},$ where $G > G'$ if $G'$ is obtained by cutting $G.$  The poset and bases are pictured in Appendix~\ref{app:hasse}. 
\begin{proof}[Proof of Theorem~\ref{thm:combotype}]
   In the poset of bases, the $K_4$ and ``house'' graph (next to the $K_4$ in Figure~\ref{fig:c246bases}) dominate every basis except for the $6$-cycle. By Lemma~\ref{lem:staticcut}, if these two are static, then so is every basis other than the $6$-cycle. 
    
    We note that the property of being static or dynamic is invariant under cyclic rotations of the vertex labels, but it is not invariant under the action of the full symmetric group. Thus we must show that for every permutation in $S_5,$ the house graph is static. (The graph $K_4$ has the full symmetric group $S_4$ as its automorphism group, so it is enough to check once). By Lemma~\ref{lem:slinvariant} the matroid of $\wedge^kZ$ only depends on the Pl\"ucker coordinates of $Z.$ Thus it suffices to parameterize the positive Grassmannian $\Gr(4,5)$ by matrices
    \[Z = \begin{bmatrix}
    1 & 0 & 0 & 0 & - x_1 \\
    0 & 1 & 0 & 0 & x_2 \\
    0 & 0 & 1 & 0 & - x_3 \\
    0 & 0 & 0 & 1 & x_4
    \end{bmatrix} \, ,\]
    with each $x_i$ positive. We then checked that for each permutation $\sigma $ in $S_5$, the polynomial given by $\det[Z_{\sigma(12)} \ Z_{\sigma(13)} \ Z_{\sigma(23)} \ Z_{\sigma(34)} \ Z_{\sigma(45)} \ Z_{\sigma(15)}]$ is a positive combination of monomials in the $x_i$, and hence itself positive.
\end{proof}

\section{Schubert hyperplanes}
\label{sec:Schubert hyperplanes}

\subsection{A general characterization}
Some of the facets of the exterior cyclic polytope are Schubert hyperplanes. In all known examples for $m=1, 2$ and $4$, only these Schubert hyperplanes contribute to the boundary of the amplituhedron. With this motivation, in Subsection \ref{subsec:A general characterization} we define a \emph{Schubert hyperplane} of the wedge power matroid and give several equivalent characterizations in Lemma \ref{lem:schuberthyp}. We characterize all Schubert hyperplanes of $C_{2,2,n}(Z)$ in Theorem~\ref{thm:nonschubert}. In Section \ref{subsec:A duality for $m=k=2$} we define a new polytope $\widetilde{C}_{2,2,n}(Z)$ by deleting the non-Schubert facets and relate its dual to the image of $Z$ under the twist map.

\label{subsec:A general characterization}
We will use $F$ to refer to a hyperplane in the matroid sense, and $H_F$ to refer to the corresponding projective linear space in $\bP(\bigwedge^k\bR^{k+m})$.
\begin{definition}
    We call a hyperplane $F$ of $W_{k,m,n}$ a \emph{Schubert hyperplane} if, for generic $Z,$ the intersection $H_F(Z) \cap \Gr(k,k+m)$ is a Schubert divisor in $\Gr(k,k+m).$
\end{definition}

We check that the property of being Schubert is constant for an open set of matrices~$Z$. Indeed, in the space $\bP(\bigwedge^{k}\bR^{k+m})^\vee$ of hyperplanes in $\bP(\bigwedge^{k}\bR^{k+m})$, a hyperplane is Schubert if and only if its coefficients satisfy the Pl\"ucker relations; see the proof of $(1) $ to $(2)$ in Lemma \ref{lem:schuberthyp}. Thus a non-Schubert hyperplane $H_F(Z)$ can only become Schubert for a Zariski closed set of matrices $Z.$  %The property of being a Schubert hyperplane is important for studying the amplituhedron. For $k=m=2,$ we will see that only the Schubert hyperplanes are needed to give a semialgebraic description of the amplituhedron.

\begin{lemma}[Schubert hyperplanes]\label{lem:schuberthyp}
    Let $F = \{Q_1, \, \ldots, \, Q_l\}$ be a hyperplane in $W_{k,m,n}$. Let $r = k+m.$ Then the following are equivalent. 
    \begin{enumerate}
        \item $F$ is a Schubert hyperplane.
        \item For $Z$ generic, the entries of the normal vector to $H_F(Z)$ satisfy the Pl\"ucker relations.
        \item For $Z$ generic, there is a $(r-k-1)$-space in $\bP^{r-1}$ meeting all the $(k-1)$-spaces $Z_{Q_1}, \ldots, Z_{Q_l}$.
    \end{enumerate}
\end{lemma}

\begin{proof}
    We first prove that $(1)$ and $(2)$ are equivalent. Suppose that $H$ is the Schubert divisor $\Gr(k,r)$ of all $(k-1)$-spaces in $\bP^{r-1}$ meeting a fixed $(r-k-1)$-space $P.$ We represent $P$ as the kernel of a $k \times r$ matrix, which we also denote $P.$ We represent an arbitrary point $L$ of $\Gr(k,r)$ as the rowspan of a $k \times r$ matrix, which we also call $L.$ Then $L$ lies on $H$ if and only if the determinant of the $k \times k$ matrix $P \cdot L^T$ vanishes. We use $p_I$ to denote maximal minors of $P$ and $q_I$ to denote maximal minors of $L.$ The Cauchy-Binet formula then gives us the equation 
    \[\sum_{I \subset \binom{[r]}{k}} p_Iq_I = 0.\]
    This proves the equivalence of (1) and (2). To prove (3), choose $N = \binom{r}{k}-1$ independent elements to span the hyperplane. By definition, they have a common transversal $(r-k-1)$-space if and only if the hyperplane is Schubert. Since the other elements $Q_i$ lie in the same hyperplane, they must also meet this common transversal.
\end{proof}

\begin{eg}\label{eg:c246schubert}
There are four combinatorial types of hyperplanes in $W_{2,2,6}$, depicted in Figure \ref{fig:c246facets}.
\begin{figure}[!h]
    \centering
    \includegraphics[width=1.0\linewidth]{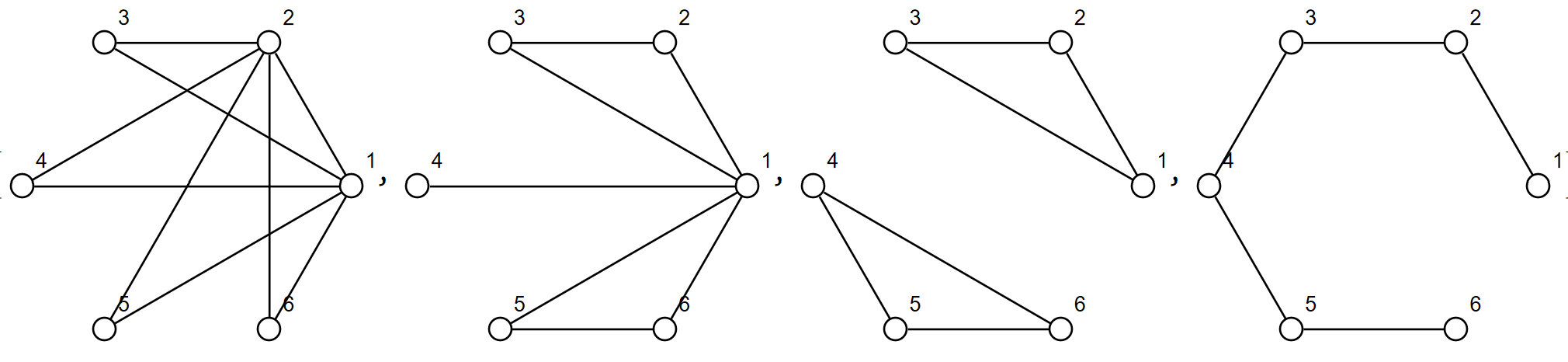}
    \caption{Hyperplanes of $W_{2,2,6}$ }\label{fig:c246facets}
\end{figure}
The number of facet-defining hyperplanes of $C_{2,2,6}$ of each type is $(6,6,3,15),$ giving a total of thirty facets, as seen in Table \ref{table:fvector}. The first three types correspond to Schubert hyperplanes in $\Gr(2,4).$ That is, they are hyperplanes parameterizing lines in $\bP^3$ meeting a fixed line. The common transversal line in each case is $12, (123) \cap (156)$ and $(123) \cap (456).$ The final graph in Figure \ref{fig:c246facets} corresponds to a non-Schubert hyperplane. The corresponding facet of $C_{2,2,6}(Z)$ is a $5$-simplex. 
\end{eg}

% In general, hyperplanes of $W_{k,m,n}$ which are also independent sets will be non-Schubert hyperplanes. 
One might conjecture from Example \ref{eg:c246schubert} that every hyperplane containing a circuit is a Schubert hyperplane. However, the following example shows that there may exist non-Schubert hyperplanes which are dependent sets in $W_{k,m,n}.$

\begin{eg}
    Let $k=2, \, m = 3$ and $n = 15.$ We label the columns of $Z$ by $1$ through $9$ and then $A$ through $F.$ Then the set consisting of \[\{12,13,14,15,16,17,18,19,1A,1B,1C,1D,1E,1F,67,89,AB,CD,EF\}\]
    is a flat in $W_{2,3,15}.$ It contains for example the circuit $\{12,13,14,15,16\}.$ The span $H$ of the corresponding vectors in $\bigwedge^2\bR^5$ has dimension $9,$ so it is indeed a hyperplane. It contains all vertices of the form $Z_1 \wedge Z_ \ast.$ However, one may check computationally that the coefficients of $H$ do not satisfy the Pl\"ucker equations for $Z$ generic. The intuition is that given nine lines in $\bP^4$, the condition that four of them meet at a point is not enough to guarantee a common transversal plane to all nine.
\end{eg}

\begin{theorem}\label{thm:nonschubert}
    The Schubert facets of the polytope $C_{2,2,n}(Z)$ are exactly the $\binom{n}{2}$ hyperplanes given by the vanishing of $\gen{Y \, \bar{i} \bar{j}}$ for $1 \leq i<j \leq n,$ where $\bar{i} \bar{j} := (i-1 i i+1) \cap (j-1 j j+1).$ Furthermore, these Schubert facets intersect transversally in $\Gr(2,4)$ for every $Z \in \rm{Mat}_{>0}(4,n)$.
\end{theorem}

\begin{proof}
Any Schubert hyperplane will contain one of the three circuits in Figure~\ref{fig:circuits4}. Thus we only need to consider the following two cases:
\begin{enumerate}
    \item three of the five lines lie in a common plane $P$, or
    \item three of the five lines meet at a point $q$.
\end{enumerate}

To complete the proof, it suffices to show that $F$ is a facet if and only if the transversal line $L$ furnished by Proposition \ref{lem:schuberthyp} is of the form $\overline{ij}$ in each case.
We show this explicitly by using the fact that if $\gen{YL}$ is a facet, then $\langle Y \, L \rangle \geq 0$ for any vertex $Y$ of $C_{2,2,n}(Z).$ We spell out the details for case 1, and case 2 follows by a similar analysis.

Call the remaining two lines $D$ and $E.$ Let us consider case 1 above, where $L=(P \cap D) \wedge (P \cap E).$ Write $P=(ijk)$, with $i<j<k$. In the following we take all indices to be ordered modulo $n$. We first show that $P$ must involve consecutive indices. For that, we use eq.~\eqref{plane_int_line} and write
\begin{equation}\label{Y_L_1}
\begin{aligned}
    \langle Y \, L \rangle = &\langle Y \, ij \rangle \langle ij \, (kD) \cap (kE) \rangle + \langle Y \, jk \rangle \langle jk \, (iD) \cap (iE) \rangle \\
    & + \langle Y \, ik \rangle \langle ik \, (jD) \cap (jE) \rangle \, .
    \end{aligned}
\end{equation}
Assume that $i+1<j<k-1$. By plugging $Y=(i+1*)$ and $Y=(j+1*)$ in \eqref{Y_L_1} for $*=i,j,k$ and requiring that \eqref{Y_L_1} is non-negative, we obtain that exactly two out of three coefficients of $\gen{Y-}$ in \eqref{Y_L_1} must vanish. Since $i,j,k$ are distinct, and so are $D$ and $E$, this is a contradiction. Therefore, either $k=j+1$, or $j=i+1$. However, since $n \geq 4$, one can repeat the same argument to show that in fact there can be only one space between $i,j,k$, so that $P=(j-1,j,j+1)=\bar{j}$. Write now $D=(ab)$ and $E=(cd)$. Using again eq.~\eqref{plane_int_line} and applying the same reasoning as above, one finds that without loss of generality $c=b=a+1$ and $d=a+2$. Therefore, $L = \overline{j,a+1}$ has the desired form. 

To prove that the Schubert hyperplanes intersect transversally for any $Z,$ we use that the twist map preserves positivity of matrices.  If $i, j, k, l$ are distinct, then $\gen{\bar{i}\bar{j}\bar{k}\bar{l}}$ is a maximal minor of $\tau(Z).$ If $Z$ is positive, then this minor is nonzero. Thus the lines $\bar{i}\bar{j}$ and $\bar{k}\bar{l}$ must not intersect. 
\end{proof}

We now describe a subset of Schubert facets of $C_{k,m,n}(Z)$ believed to form part of the algebraic boundary of the amplituhedron. In \cite{Arkani_Hamed_2018} the authors conjecture that the amplituhedron $\mathcal{A}_{k,m,n}(Z) \subset \Gr(k,k+m)$ has the following semialgebraic description:
\begin{align}
&(-1)^k \langle Y\, 1 \, (i_1 i_{1}+1) \dots (i_{\frac{m-1}{2}} i_{\frac{m-1}{2}} + 1) \rangle \, , \ \langle Y\, (i_1 i_{1}+1) \dots (i_{\frac{m-1}{2}} i_{\frac{m-1}{2}} + 1) n \rangle >0 \quad m \ \text{odd}, \label{facets_m_ood} \\
&\langle Y \,   (i_1 i_{1}+1) \dots (i_{\frac{m}{2}} i_{\frac{m}{2}} + 1) \rangle >0  \quad m \ \text{even}, \label{facets_m_even} \\
&(\langle Y \, 12 \dots m-1 m) \rangle , \,  \dots , \, \langle Y\,  12 \dots n-1 n) \rangle) \text{ has $k$ sign flips}.
\end{align}
To our knowledge, this semialgebraic description of the amplituhedron has been proven only for $m=2$ in \cite[Theorem 5.1]{Parisi:2021oql}. For this case, the algebraic boundary is known to consist of the $n$ Schubert divisors $\langle Y \, ii+1 \rangle = 0$. This has been proven for $k=2$ in \cite[Proposition 3.1]{Ranestad_2024}, but a similar proof works for higher~$k$. Furthermore, in \cite[Corollary 8.8]{Even-Zohar:2021sec} the authors showed that the algebraic boundary of the $m=4$ amplituhedron is given by the Schubert divisors $\langle Y \, ii+1jj+1 \rangle = 0$.

\begin{proposition}\label{prop:C_facets}
    Equations \eqref{facets_m_ood} and \eqref{facets_m_even}, for $m$ odd or even respectively, define facets of $C_{k,m,n}(Z)$.
\end{proposition}

\begin{proof}
    By positivity of $Z$, all vertices of $C_{k,m,n}$ lie on the same side of each hyperplane in \eqref{facets_m_ood} and \eqref{facets_m_even}. We are left to show that each such hyperplane contains $\binom{k+m}{k}-1$ vertices of $C_{k,m,n}$ in general position. For general $n$, the number of vertices lying on one such hyperplane is
    \begin{equation}
        \binom{n}{k} - \binom{n-m}{k} \, .
    \end{equation}
    For $n=k+m$ this is exactly equal to $\binom{k+m}{k}-1$. Since $n$ is at least $k+m,$ it is enough to consider vertices labeled by subsets of $[k+m]$. These are linearly independent since $Z$ is positive. %In fact, the $(k+m)$-minor of $Z$ containing the first columns is positive, showing that $Z$ is an isomorphism when restricted to the first $k+m$ standard basis vectors. 
\end{proof}

\subsection{A duality for $m=k=2$}
\label{subsec:A duality for $m=k=2$}
We define a new polytope which only involves the Schubert hyperplanes. We are giving up some combinatorial control of the vertices for a better understanding of the facets.

\begin{definition}
    The \emph{Schubert exterior cyclic polytope} $\widetilde{C}_{k,m,n}(Z)$ is obtained from the exterior cyclic polytope $C_{k,m,n}(Z)$ by deleting all facet inequalities corresponding to non-Schubert facets. 
\end{definition}

\begin{eg}
    When $m=1$ every hyperplane is Schubert, and hence $C_{k,1,n}(Z) = \widetilde{C}_{k,1,n}(Z)$. The same is true for $k=1$.
\end{eg}
In principle, when deleting facet inequalities one could end up with a polyhedron which is unbounded in every affine chart, and hence not a polytope. We show that for most parameters $k,m,n$ we actually obtain a polytope. In the following we view normal vectors as points in $(\bP^{N})^\vee$.

\begin{lemma}\label{lem:polyhedron}
    Consider a projective polyhedron $P \subset \bP^
    N.$  Then $P$ is bounded in some affine chart if and only if the normal vectors to the facets do not lie on a common hyperplane in $(\bP^{N})^\vee.$
\end{lemma}
\begin{proof}
The polyhedron is unbounded in every affine chart if and only if the cone over it in $\bR^{N+1}$ is not proper, i.e. contains some linear space $L.$ In this case every normal vector lies in $\bP(L^\perp).$
\end{proof}

Proposition \ref{prop:C_facets} tells us that $C_{k,m,n}(Z)$ has at least $\binom{n}{m/2}$ Schubert facets. When this number is at least $\binom{k+m}{k},$ the $\widetilde{C}_{k,m,n}(Z) \subset \bP(\bigwedge^k\bR^{k+m})$ is a polytope by Lemma \ref{lem:polyhedron}. Indeed, the normal vectors span a space of projective dimension $\binom{n}{m/2}-1$. In the case $k=m=2,$ Lemma \ref{thm:nonschubert} produces $\binom{n}{2}$ hyperplanes and $\widetilde{C}_{2,2,n}(Z)$ lives in $\bP^5,$ so we indeed get a polytope. %\lizzie{following prop is not proven .. why is it true? it's also not used anywhere}

% It follows by definition that the Schubert exterior polytope also contains information about the dual polytope of $C_{k,m,n}(Z)$, where we denote the latter by $C_{k,m,n}(Z)^{*}$.
%\begin{proposition}\label{prop:C_tilde_f=C_star_v}
%     The facets of $\widetilde{C}_{k,m,n}(Z)$ are (projectively) dual to vertices of $C_{k,m,n}(Z)^*$ lying on $\Gr(k,k+m)$. Each of the latter is a point in the intersection of at least $\binom{k+m}{k}-1$ Schubert divisors of the form
%     \begin{equation}
%         \big\{Y \in \Gr(k,k+m) \, : \, \langle Y \, i_1 \, \cdots \, i_k \rangle = 0 \big\} \, , \qquad (i_1 \, \cdots \, i_k) \in \binom{[n]}{k} \, .
%     \end{equation}
% \end{proposition}

We study the Schubert exterior cyclic polytope $\widetilde{C}_{2,2,n}(Z)$ and its relation to $C_{2,2,n}(Z)$. For that, we find it useful to present a result appearing in \cite[Section 14]{Arkani_Hamed_2018}, whose proof can be found in \cite[Corollary 4.10]{Even-Zohar:2023del}. In that reference it is shown that eq.~\eqref{magic_det} is a cluster variable for $\Gr_{>0}(4,n)$.

\begin{proposition}[Positive determinants for $k=m=2$]
\label{prop_magic_det}
    Let $Z \in {\rm Mat}_{> 0}(4,n)$. Let $i_1 < i_2 < i_3 $ and $j_1 < j_2 < j_3$ be elements in $[n]$ such that $i_r \leq j_r$ for every $r=1,2,3$ and equality does not hold for all $r$. Let $a<b$ in $[n] \setminus ([i_1+1, \dots, i_3-1] \cup [j_1+1, \dots, j_3-1])$. Then,
    \begin{equation}\label{magic_det}
        \det \begin{pmatrix}
            \langle a i_1 i_2 i_3 \rangle & \langle a j_1 j_2 j_3 \rangle \\
            \langle b i_1 i_2 i_3 \rangle & \langle b j_1 j_2 j_3 \rangle
        \end{pmatrix} = \langle ab (i_1i_2i_3) \cap (j_1 j_2 j_3) \rangle  \geq 0 \ .
    \end{equation}
\end{proposition}

The polytope $\widetilde{C}_{2,2,n}(Z)$ contains $C_{2,2,n}(Z)$ by definition.
Moreover, by Theorem~\ref{thm:nonschubert} and Proposition~\ref{prop_magic_det} we have that
\begin{equation}\label{eq:C_tilde_ineq}
    \widetilde{C}_{2,2,n}(Z) = \{Y \in \mathbb{P}^5 \, : \,  \langle Y \, \overline{ij} \rangle \geq 0 \, , \quad \forall \, 1 \leq i<j \leq n\} \, .
\end{equation}
The polytope $\widetilde{C}_{2,2,n}$ is related to the dual polytope of $C_{2,2,n}$, which we denote by $C_{2,2,n}^{*}$, as follows.
\begin{proposition}\label{prop:parity dual}
    We have that
\begin{equation}\label{eq_C_tildeC_duality}
    \widetilde{C}_{2,2,n}(Z) =  C_{2,2,n}(W)^* \, ,
\end{equation}
where $W=\tau(Z)$ is the twist of $Z$, according to Definition~\ref{def:twist_map}. 
\end{proposition}
\begin{proof}
    By definition of the exterior cyclic polytope, we have that $C_{2,2,n}(W) \subset \mathbb{P}^5$ is equal to the convex hull of the $\binom{n}{2}$ points $W_i \wedge W_j = \overline{ij}$. Note that $\overline{ii+1} = (ii+1)$, and hence by Proposition~\ref{prop_magic_det} we have that $\overline{ij}$ is an inward-pointing normal vector of $C_{2,2,n}(W)^* \subset \mathbb{P}^5$. Thus $C_{2,2,n}(W)^*$ is exactly given by the right hand side of eq.~\eqref{eq:C_tilde_ineq}, and hence it is equal to $\widetilde{C}_{2,2,n}(Z)$.
\end{proof}

By eq.~\eqref{eq_C_tildeC_duality} it follows that the vertices of $\widetilde{C}_{2,2,n}(Z)$ lying on $\Gr(2,4)$ correspond to Schubert facets of $C_{2,2,n}(W)$, which by Lemma~\ref{thm:nonschubert} are projectively dual to $ (ij)$ for $1 \leq i<j \leq n$. Note that here we also used that for $W=\tau(Z) \in {\rm Mat}_{>0}(4,n)$ we have that $\tau(W)_i \wedge \tau(W)_j = (ij)$ as elements in $\mathbb{P}(\bigwedge^2 \mathbb{R}^4)$, which can be easily verified. We have therefore shown that
\begin{equation}\label{Ex_C_tildeB_C}
       \Gr(2,4) \cap {\rm vert}(\widetilde{C}_{2,2,n})  =  {\rm vert}(C_{2,2,n}) \,,
\end{equation}
where ${\rm vert}(P)$ denotes the set of vertices of the polytope $P$.

\subsection{Schubert facets for $k=3$ and $m=2$}
\label{subsec:Some results for higher $k$ and $m$}

Let us consider the case of $k=3$ and $m=2$. In the following, we determine computationally the Schubert facets of $C_{3,2,6}(Z)$.
Recall that the exterior cyclic polytope $C_{3,2,n}(Z)$ is defined as the convex hull of the $\binom{n}{3}$ points $(ijk)$ in $ \mathbb{P}(\bigwedge^3 \mathbb{R}^5) \cong \mathbb{P}^9$ with $1 \leq i<j<k \leq n$.

\begin{eg}\label{eg:k=3_n=6}
    Let $n=6$. We compute the $f$-vector of $C_{3,2,6}(Z)$ to be
    \begin{equation}
        (20, \, 160, \, 675, \, 1659, \, 2469, \, 2227, \, 1173, \, 327, \, 38, \, 1) \, .
    \end{equation}
    % confirmed with non Vandermonde example 

    Among the 38 facets, 20 are Schubert, which we check computationally by checking whether their coefficients satisfy Pl\"ucker relations as per Theorem \ref{thm:nonschubert}. Each Schubert facet is supported by a hyperplane in $\mathbb{P}^9$ projectively dual to a point in $\Gr(2,5)$. Among these points, we find the six given by $(ii+1)$ from Proposition~\ref{prop:C_facets}. Each associated Schubert facet contains 16 vertices. 
    The remaining 14 Schubert facets come in two families. The first consists of cyclic shifts of  the facet projectively dual to $(345) \cap (6123)$, containing 14 vertices whose complement in $\binom{[6]}{3}$ is 
    \begin{equation}
    (1 2 4) \, , \ (1 2 5) \, , \ (1 4 6) \, , \ (1 5 6) \, , \ (246) \, , \  (256).
    \end{equation}This family contains 12 facets.
    The second family consists of two facets, for which one representative is $ (12) \cap (3456) \wedge (56) \cap (1234) = (34) \cap (5612) \wedge (12) \cap (3456)$, containing the 12 vertices 
    \begin{equation}
    (1 2 3) \text{ and its cyclic shifts} \, , \ (1 2 4) \, , \ (1 25) \, ,  \ (1 34) \, , (346) \, , \  (256) \, , \ (356) \, .
    \end{equation}
We also computed the $f$-vector of the Schubert exterior polytope $\widetilde{C}_{3,2,6}$ and found it to be the same as that of $C_{3,2,6}^*$. Moreover, we checked that the normal vectors to the Schubert facets correspond to elements $W_i \wedge W_j \wedge W_k$ with $1 \leq i<j<k \leq n$, where $W= \tau(Z)$ is the image of $Z$ under the twist map of Definition~\ref{def:twist_map}. For example,
\begin{equation}
    (ii+1) = (i-2,i-1,i,i+1) \cap (i-1,i,i+1,i+2) \cap (i,i+1,i+2,i+3) \, , \quad i \in [n] \, ,
\end{equation}
corresponds to $W_{i-1} \wedge W_{i} \wedge W_{i+1}$ under the isomorphism $\mathbb{P}(\bigwedge^2 \mathbb{R}^5) \cong \mathbb{P} (\bigwedge^3 \mathbb{R}^5)$ induced by the standard inner product on $\mathbb{R}^5$. Similarly, $(i-1,i,i+1,i+2) \cap (j-1,j,j+1) $ for $i,j \in [n]$ with $|i-j|>1$ corresponds to $W_i \wedge W_{j-1} \wedge W_j$. 
\end{eg}

%\begin{eg}\label{eg:k=3_n=7}
%    Let $n=7$. Among the 1020 facets of $C_{3,2,7}(Z)$, 35 are Schubert. Among these, we find the seven ones dual to $( ii+1) $. Then, there are 14 facets dual to $(1234) \cap (456)$ and $(123) \cap (3456)$, together with their cyclic orbits. There are 7 facets dual to $(1234) \cap (56) \wedge (1234) \cap (67)$, and other 7 facets dual to $(1234) \cap (67) \wedge (12) \cap (4567)$, including their cyclic orbits respectively.
%\end{eg}

Based on Example~\ref{eg:k=3_n=6}, we conjecture that $\widetilde{C}_{3,2,n}$ for $n \geq 5$ is a polytope with $\binom{n}{3}$ Schubert facets, each of the latter dual to some $W_i \wedge W_j \wedge W_k$ for $1 \leq i<j<k \leq n$.
More generally, we pose the following question.

\begin{question}\label{question:C_duality}
    Is the analogue of eq.~\eqref{eq_C_tildeC_duality} true for higher $m$ and $k$? 
\end{question}

\section{Convexity of amplituhedra}\label{sec:Convexity of Amplituhedra}

In the following, we fix a matrix $Z \in \text{Mat}_{> 0}(k+m,n)$ for $n \geq k+m$, and omit the dependence on $Z$ in the notation. Our first theorem concerns the convex hull of the amplituhedron in Pl\"ucker space.

\begin{proposition}\label{prop:amphull}
    The convex hull of the amplituhedron $\cA_{k,m,n}$ in $\mathbb{P}^{\binom{k+m}{k}-1}$ is equal to the exterior cyclic polytope $C_{k,m,n}.$
\end{proposition}
\begin{proof}
Because the non-negative Grassmannian is contained in the standard projective simplex, we immediately get the inclusion 
\begin{equation}\label{eq_easy_inclusion}
    \mathcal{A}_{k,m,n} \subset \Gr(k,k+m) \cap C_{k,m,n} \, .
\end{equation}
Monotonicity of the convex hull gives $\text{conv}(\cA_{k,m,n}) \subset C_{k,m,n}.$ The other direction follows because the vertices of $C_{k,m,n}$ are contained in $\cA_{k,m,n}.$ 
\end{proof}

In the following, we want to determine which amplituhedra are extendably convex. In light of Proposition~\ref{prop:amphull}, this amounts to asking for which parameters the inclusion in eq.~\eqref{eq_easy_inclusion} is an equality.
\begin{question}\label{question_int}
    For which values of $k,m,n$ is $\mathcal{A}_{k,m,n}$ extendably convex?
\end{question}
Note that in phrasing this question we dropped the dependence on $Z \in {\rm Mat}_{>0}(k+m,n)$, but it is not clear if this is a reasonable assumption. We expect the answer to Question~\ref{question_int} to depend on $m.$ For $m=0$ the amplituhedron and the exterior cyclic polytope are both equal to $\mathbb{P}^{0} $, and therefore coincide. However, for the case of $m=1$ it follows from the discussion in Section~\ref{subsec:The case $m=1$} that the amplituhedron is not extendably convex. In general, we expect that the amplituhedron is not extendably convex for odd $m$. However, for $m=2$ and $k=2$ we will prove in Section~\ref{subsec:Convexity for $k = m = 2$} that eq.~\eqref{eq_easy_inclusion} is in fact an equality.

If the amplituhedron $\mathcal{A}_{k,m,n}$ is extendably convex, its boundary is equal to $\Gr(k,k+m) \cap \partial C_{k,m,n}$. In particular, the  algebraic boundary would consist of only linear sections in the Grassmannian. It has been proven that this is the case for $m=1,2$ and $4$. However, for $m=6$, the boundary of the amplituhedron contains higher degree components~\cite{higher_m_ampl}.

With this motivation, for the rest of the section we explore the question: which boundaries of $C_{k,m,n}$ intersect the amplituhedron? In other words, what is the linear part of the boundary of $\cA_{k,m,n}$? We first explore this question for all of the hyperplanes in the matroid~$W_{k,m,n}$. 
\begin{definition}
    The \textit{$k$-th exterior arrangement} of type $(m,n)$, denoted  $H_{k,m,n}(Z)$, is the set of hyperplanes in $\bigwedge^k\bR^n$ spanned by the columns of $\wedge^kZ.$ 
\end{definition}
In other words, $H_{k,m,n}(Z)$ is the discriminantal arrangement of the columns of $\wedge^kZ.$ 

\begin{definition}[Positroid hyperplane]
    Consider a hyperplane $F = \{I_1, \, \ldots, \, I_r\}$ in $W_{k,m,n}$, where each $I_i$ is in $\binom{[n]}{k}$. We call $F$ a \emph{positroid hyperplane} if the elements of $F$ form the bases of a positroid of rank $k$ on $n$ elements.
\end{definition}

\begin{lemma}\label{lem:poshyperplane}
    Let $H_F$ be a facet hyperplane of $C_{k,m,n}.$ If the amplituhedron $\cA_{k,m,n}$ intersects the interior of $H_F,$ then $F$ is a positroid hyperplane. 
\end{lemma}

\begin{proof}
    Consider the pre-image of the boundary of $C_{k,m,n}$ in the  standard simplex in $\bP(\bigwedge^k\bR^n).$ This is a simplicial complex $\Delta$ in the boundary of the standard simplex. Taking the pre-image of a single facet $P_F$ gives a maximal simplex in $\Delta.$ Concretely, this pre-image consists of non-negative sums of $e_I,$ such that $I$ is an element of $F.$ Then the amplituhedron intersects the interior of $P_F$ if and only if the non-negative Grassmannian $\Gr_{\geq 0}(k,n)$ intersects the interior of the hyperplane $\widetilde{H}_F := \text{span}\{e_I \, : \, I \in F\}.$ 
    
    Let us determine in which cases this happens. Setting the Pl\"ucker variables $\{p_I \, : \, I \notin F\}$ equal to zero in $\Gr(k,n)$ may impose that other Pl\"ucker variables are zero, because of the Pl\"ucker relations. If this happens then $\Gr(k,n) \cap \widetilde{H}_F$ has dimension less than that of $\widetilde{H}_F.$ 
    
    Let $J$ be the Pl\"ucker ideal. The dimension of $\Gr(k,n) \cap \widetilde{H}_F$ is equal to that of $\widetilde{H}_F$ whenever $\gen{p_I \, : \, I \notin F} \oplus J = \gen{p_I \, : \, I \notin F}.$ But this is true if and only if the elements of $F$ form the bases of a matroid. Suppose this holds. Since we started with a facet of $C_{k,m,n},$ we know that $\widetilde{H}_F$ lies on the boundary of the non-negative simplex. This additional condition is satisfied if and only if the elements of $F$ further form the bases of a positroid.
\end{proof}

Thus instead of studying all hyperplanes in $H_{k,m,n},$ we consider only the positroid hyperplanes.

\subsection{Linear boundaries for $k=2$ and $m = n-4$}
\label{subsec:Linear boundaries for $k=2$ and $m = n-4$}

 %We will use the structure of the hyperconnectivity matroid to prove the following proposition for $d=2$ \elia{What is $d$?}, which corresponds to $m = n-4.$
We now characterize a certain set of linear boundaries of the amplituhedron $\cA_{2, n-4,n}$. Curiously, the amplituhedron for these parameters is mapped via the parity duality map~\cite{Galashin_2020} to the amplituhedron $\cA_{2,2,n}$, which we discuss in the next section. 

\begin{proposition}[Linear boundaries]\label{prop:boundcapamp}
If the amplituhedron $\cA_{2, n-4,n}$ intersects the boundary of a facet hyperplane of $C_{2,n-4,n},$ then that hyperplane is of the form $F_{ijkl} := \text{conv} \Big\{(ab) \ : \ ab \notin \binom{{i,j,k,l}}{2} \Big\}$ corresponding to the complements of complete graphs on four vertices.
\end{proposition}
The facets $F_{ijkl}$ are given by the vanishing of eq.~\eqref{facets_m_even} and \eqref{facets_m_ood} for $n$ even and odd, respectively. For example, if $n$ is even, the hyperplane in $\mathbb{P}^{N}$ with $N=\binom{n-2}{2}-1$ defined by the vanishing locus of eq.~\eqref{facets_m_even} contains all vertices $(ab)$ except those for which both $a$ and $b$ belong to $[n] \setminus \{i_1,i_{1}+1, \dots, i_{\frac{m}{2}}, i_{\frac{m}{2}} + 1\}$. As $m=n-4$, the latter consists of four points $\{i,j,k,l\}$. Hence, the vanishing of eq.~\eqref{facets_m_even} cuts out the facet $F_{ijkl}$.

To prove Proposition~\ref{prop:boundcapamp} we use the following lemma.
As before, we identify flats of $W_{2,m,n}$ with graphs on $n$ vertices. We describe positroid hyperplanes combinatorially in terms of these graphs. 

\begin{lemma}\label{lem:posgraphs}
    Suppose that $F$ is a positroid hyperplane of $W_{2,m,n}$. Then every induced subgraph of $(ij)_{ij \in F}$ on four vertices is among the list in Figure \ref{fig:posgraphs}.
\end{lemma}
\begin{proof}
We represent rank $2$ matroids on $[n]$ as graphs on $n$ vertices with a specified cyclic ordering. If such a matroid is a positroid, then every restriction of it to four elements  $i,j,k,l$ is a positroid with respect to the cyclic order inherited from $[n].$ To every positroid $M$ on four elements we associate a graph $(ij)_{ij \in M},$ obtaining the list in Figure \ref{fig:posgraphs}. Thus if $G$ is a graph representing a positroid hyperplane, then every induced $4$-subgraph of $G$ is in this list.
\end{proof}

\begin{figure}[!h]
    \centering
    \scalebox{0.8}{\input{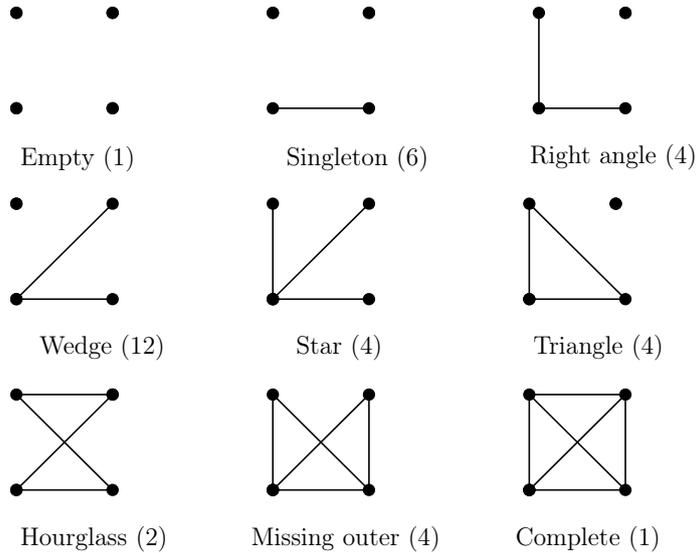}}
    \caption{Positroids on four elements, as graphs}
    \label{fig:posgraphs}
\end{figure}

By a standard argument (see \cite[Proposition 2.1.6]{Oxley}), $F$ is a hyperplane in a matroid whenever $\binom{[n]}{2}\setminus F$ is a circuit in the dual matroid. We define a \emph{positroid circuit} in $\cH_d(n)$ to be the complement of a positroid hyperplane in $W_{2,m,n}$, where $d = n-m-2$. By Lemma \ref{lem:posgraphs}, positroid circuits are circuits of $\cH_d(n)$ such that every induced four-graph is the complement of a graph in \ref{fig:posgraphs}. 

\begin{proof}[Proof of Proposition~\ref{prop:boundcapamp}]
We first show that the linear space $H_{ijkl}$ corresponding to the set $F_{ijkl}$ is indeed a hyperplane in $\bP(\bigwedge^{2}\bR^{n-2}).$ This is because it spans the Schubert hyperplane $\{P  \in \Gr(2,n) \ : \ P \text{ meets } [n]\setminus \{i,j,k,l\} \}.$ The complement of the hyperplane $F_{ijkl}$ is the complete graph on $K_4$ vertices $\{i,j,k,l\}.$ Thus $K_4$ is a circuit in $\cH_2(n).$ By Lemma \ref{lem:poshyperplane}, it suffices to show that these copies of $K_4$ are the only positroid circuits in $\cH_2(n).$ 

Suppose we have a graph $G$ on $n \geq 5$ vertices which is a circuit in $\cH_2(n)$. By Corollary 5.5 in \cite{kalai}, every vertex in the graph has degree at least $3.$ %We have drawn the positroid graph complements with $3, 4,$ or $5$ vertices in Figure \ref{fig:poscomplement}.
We claim that if every size four induced subgraph of $G$ is the complement of a graph in Figure \ref{fig:posgraphs}, then $G$ must be $K_4.$

Suppose for contradiction that $G$ is not $K_4.$ By degree considerations, there must be a subgraph of $G$ which is a four-cycle. But the only complements of graphs in Figure \ref{fig:posgraphs} which contain a four-cycle are $K_4$ and $K_4 \setminus e$ for some edge $e$. 
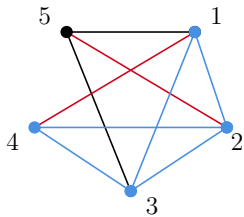
\begin{figure}[!h]
    \centering
    \scalebox{0.8}{\tikzset{every picture/.style={line width=0.75pt}} %set default line width to 0.75pt        

\begin{tikzpicture}[x=0.75pt,y=0.75pt,yscale=-1,xscale=1]

%Straight Lines [id:da3999646625487603] 
\draw [color={rgb, 255:red, 208; green, 2; blue, 27 }  ,draw opacity=1 ]   (300,80) -- (400,140) ;
%Straight Lines [id:da9742857212679529] 
\draw    (300,80) -- (340,180) ;
%Straight Lines [id:da7815041795804303] 
\draw [color={rgb, 255:red, 208; green, 2; blue, 27 }  ,draw opacity=1 ]   (280,140) -- (380,80) ;
%Straight Lines [id:da6731047102186258] 
\draw    (300,80) -- (380,80) ;
\draw [shift={(380,80)}, rotate = 0] [color={rgb, 255:red, 0; green, 0; blue, 0 }  ][fill={rgb, 255:red, 0; green, 0; blue, 0 }  ][line width=0.75]      (0, 0) circle [x radius= 3.35, y radius= 3.35]   ;
\draw [shift={(300,80)}, rotate = 0] [color={rgb, 255:red, 0; green, 0; blue, 0 }  ][fill={rgb, 255:red, 0; green, 0; blue, 0 }  ][line width=0.75]      (0, 0) circle [x radius= 3.35, y radius= 3.35]   ;
%Straight Lines [id:da08548843263831063] 
\draw [color={rgb, 255:red, 74; green, 144; blue, 226 }  ,draw opacity=1 ][fill={rgb, 255:red, 74; green, 144; blue, 226 }  ,fill opacity=1 ]   (380,80) -- (400,140) ;
\draw [shift={(400,140)}, rotate = 71.57] [color={rgb, 255:red, 74; green, 144; blue, 226 }  ,draw opacity=1 ][fill={rgb, 255:red, 74; green, 144; blue, 226 }  ,fill opacity=1 ][line width=0.75]      (0, 0) circle [x radius= 3.35, y radius= 3.35]   ;
\draw [shift={(380,80)}, rotate = 71.57] [color={rgb, 255:red, 74; green, 144; blue, 226 }  ,draw opacity=1 ][fill={rgb, 255:red, 74; green, 144; blue, 226 }  ,fill opacity=1 ][line width=0.75]      (0, 0) circle [x radius= 3.35, y radius= 3.35]   ;
%Straight Lines [id:da34325020347048874] 
\draw [color={rgb, 255:red, 74; green, 144; blue, 226 }  ,draw opacity=1 ][fill={rgb, 255:red, 74; green, 144; blue, 226 }  ,fill opacity=1 ]   (340,180) -- (280,140) ;
\draw [shift={(280,140)}, rotate = 213.69] [color={rgb, 255:red, 74; green, 144; blue, 226 }  ,draw opacity=1 ][fill={rgb, 255:red, 74; green, 144; blue, 226 }  ,fill opacity=1 ][line width=0.75]      (0, 0) circle [x radius= 3.35, y radius= 3.35]   ;
\draw [shift={(340,180)}, rotate = 213.69] [color={rgb, 255:red, 74; green, 144; blue, 226 }  ,draw opacity=1 ][fill={rgb, 255:red, 74; green, 144; blue, 226 }  ,fill opacity=1 ][line width=0.75]      (0, 0) circle [x radius= 3.35, y radius= 3.35]   ;
%Straight Lines [id:da1807043736685987] 
\draw [color={rgb, 255:red, 74; green, 144; blue, 226 }  ,draw opacity=1 ][fill={rgb, 255:red, 74; green, 144; blue, 226 }  ,fill opacity=1 ]   (400,140) -- (340,180) ;
\draw [shift={(340,180)}, rotate = 146.31] [color={rgb, 255:red, 74; green, 144; blue, 226 }  ,draw opacity=1 ][fill={rgb, 255:red, 74; green, 144; blue, 226 }  ,fill opacity=1 ][line width=0.75]      (0, 0) circle [x radius= 3.35, y radius= 3.35]   ;
\draw [shift={(400,140)}, rotate = 146.31] [color={rgb, 255:red, 74; green, 144; blue, 226 }  ,draw opacity=1 ][fill={rgb, 255:red, 74; green, 144; blue, 226 }  ,fill opacity=1 ][line width=0.75]      (0, 0) circle [x radius= 3.35, y radius= 3.35]   ;
%Straight Lines [id:da29833352695449256] 
\draw [color={rgb, 255:red, 74; green, 144; blue, 226 }  ,draw opacity=1 ][fill={rgb, 255:red, 74; green, 144; blue, 226 }  ,fill opacity=1 ]   (380,80) -- (340,180) ;
%Straight Lines [id:da8420700250631376] 
\draw [color={rgb, 255:red, 74; green, 144; blue, 226 }  ,draw opacity=1 ][fill={rgb, 255:red, 74; green, 144; blue, 226 }  ,fill opacity=1 ]   (280,140) -- (400,140) ;

% Text Node
\draw (388,63.4) node [anchor=north west][inner sep=0.75pt]    {$1$};
% Text Node
\draw (401,142.4) node [anchor=north west][inner sep=0.75pt]    {$2$};
% Text Node
\draw (348,182.4) node [anchor=north west][inner sep=0.75pt]    {$3$};
% Text Node
\draw (261,142.4) node [anchor=north west][inner sep=0.75pt]    {$4$};
% Text Node
\draw (281,63.4) node [anchor=north west][inner sep=0.75pt]    {$5$};

\end{tikzpicture}}
    \caption{Building up the graph $G$}
    \label{fig:k4proof}
\end{figure}
Thus $G$ must contain a copy of $K_4 \setminus e,$ depicted in blue in Figure \ref{fig:k4proof}. Furthermore, the degree two vertex $1$ must connect to a fifth vertex $5$ somewhere in $G$. The induced graph $1345$ contains a $3$-walk and thus must contain a fourth edge $14$ or $35$; since we are avoiding $K_4,$ we must have $35.$ By the same logic, $25$ is in the complement of $G.$ But then the restriction to $1,2,4,5$ is either a $3$-walk or a $4$-walk, neither of which is a complement of a graph in Figure \ref{fig:posgraphs}.
\end{proof}
The obstruction to a version of Proposition \ref{prop:boundcapamp} for other $m = n-d-2$ is that we do not have a good understanding of circuits in the hyperconnectivity matroid. By Corollary~5.5 in \cite{kalai}, every circuit in $\cH_d(n)$ has edge connectivity at least $d+1.$ So as $d$ gets large, the analysis of circuits in $\cH_d(n)$ becomes much more involved.

\subsection{Convexity for $k = m = 2$}
\label{subsec:Convexity for $k = m = 2$}

In this section we prove that the amplituhedron for $k=m=2$ is extendably convex.
For that, we need a sequence of lemmas to verify the assumptions in Lemma~\ref{lemma:inters with polytope}. We start from the regularity assumption.
\begin{lemma}\label{lemma:C_vert_reg}
    The vertices ${\rm vert}(C_{k,m,n}) $ are all regular points in $\Gr(k,k+m) \cap C_{k,m,n}$. 
\end{lemma}

\begin{proof}
    This follows from the fact that ${\rm vert}(C_{k,m,n}) $ is contained in $\mathcal{A}_{k,m,n}$ and the relative interior of $\mathcal{A}_{k,m,n}$ in $\Gr(k,k+m)$ is contained in $\Gr(k,k+m) \cap \inter(C_{k,m,n})$. 
\end{proof}

\begin{lemma}\label{leamma:reg_C_{2,2,n}}
    The semialgebraic set $\Gr(2,4) \cap C_{2,2,n}$ is regular in $\Gr(2,4)$. The same is true for $\Gr(2,4) \cap \widetilde{C}_{2,2,n}$.
\end{lemma}

\begin{proof}
Since the projective dual variety of $\Gr_\mathbb{C}(2,4)$ is equal to $\Gr_\mathbb{C}(2,4)$, any hyperplane section of $\Gr(2,4)$ is singular if and only if the hyperplane is Schubert. In Lemma \ref{thm:nonschubert} we characterized all Schubert hyperplanes of $C_{2,2,n}$. Each of these is projectively dual to a point $\overline{ij} \in \Gr(2,4)$ for some $1 \leq i < j \leq n$. Each such Schubert divisor is singular only at one point, given by $\overline{ij}$. Therefore the singular points of Schubert hyperplanes of $C_{2,2,n}$ intersected with $\Gr(2,4)$ are all $\overline{ij}$ for $1 \leq i < j \leq n$. We show that $\overline{ij}$ is in $\partial C_{2,2,n}$ if and only if $j=i+1$, modulo $n$, in which case $\overline{ij} = (ii+1) $. In fact, if $|j-i|>1$, then using \eqref{plane_int_plane} we compute
\begin{equation}
    \langle \overline{ij} , \, \overline{i-1i+1} \rangle  = \langle i-1 ,i+1 ,i , i+2 \rangle \langle i , j-1,j,j+1 \rangle \langle i-2,i-1,i , i+1 \rangle < 0 \, .
\end{equation}
By eq.~\eqref{eq:C_tilde_ineq} it follows that $\overline{ij} \notin C_{2,2,n}$ if $|j-i|>1$. On the other hand, $(ii+1) \in \partial C_{2,2,n}$ and it is in fact a vertex of $C_{2,2,n}$, and hence a regular point of $\Gr(2,4) \cap C_{2,2,n}$ by Lemma~\ref{lemma:C_vert_reg}.

Note that since $\Gr(2,4)$ has codimension one in $\mathbb{P}^5$, Lemma \ref{lemma:reg_X_P} applies to all faces of $C_{2,2,n}$ of dimension greater equal or than one. We also use that for every face $F$ of $C_{2,2,n}$ we have that
\begin{equation}
    {\rm Sing}(X \cap \overline{F}) \subset \bigcup_i {\rm Sing}(X \cap H_i) = \{(ii+1) \, : \, i=1,\dots,n\} \, ,
\end{equation}
where $\overline{F}$ denotes the Zariski closure of $F$ and the union ranges over all facet-defining hyperplanes $H_i$ of $C_{2,2,n}.$ The last equality follows from the previous analysis. The proof is therefore concluded by Lemma \ref{lemma:C_vert_reg}.

The exact same argument works also for $\Gr(2,4) \cap \widetilde{C}_{2,2,n}$. The fact that the vertices are regular also in this case follows from the containment of $C_{2,2,n}$ in $\widetilde{C}_{2,2,n}$.
\end{proof}

We now check the connectedness assumption of Lemma~\ref{lemma:inters with polytope} in our setting. We first need the following result.

\begin{lemma}\label{lemma:exist_negative_ij}
    We have that
    \begin{equation}
        \{Y \in \mathbb{P}^5 \, : \, \langle Y \, ij \rangle > 0 \, , \ \forall 1 \leq i<j \leq n\} \cap \{Y \in \mathbb{P}^5 \, : \, \langle Y \, \overline{ij} \rangle > 0 \, , \ \forall 1 \leq i<j \leq n\} = \emptyset \,. 
    \end{equation}

\end{lemma}

\begin{proof}
    %We regard $\langle Y \, ab \rangle $ for $1 \leq a < b \leq n$ as the $(ab)$-Plücker coordinate on $\bigwedge^2 \mathbb{R}^n \cong \mathbb{R}^{\binom{n}{2}}.$ In particular, consider the map 
    %\begin{equation}
    %\begin{aligned}
    %    \overline{\psi} \, \colon \, \Gr(&2,4) \rightarrow \bigwedge^2 \mathbb{R}^n \, , \\
    %    & Y \mapsto \langle Y \, \overline{ab} \rangle \, .
    %\end{aligned}
    %\end{equation}
    %This map is well-defined and injective~\cite[Lemma 1]{Lam:2024gyg}.
    Consider the map $T,$ which depends on $Z,$ given by
    \begin{equation}
    \begin{aligned}
    T \, \colon \, \mathbb{R}^{\binom{n}{2}} & \rightarrow \mathbb{R}^{\binom{n}{2}} \, , \\
    \sum_{1\leq a<b \leq n} \langle Y \, ab \rangle \, e_{ab} & \mapsto \sum_{1\leq a<b \leq n}  \langle Y \, \overline{ab} \rangle \, e_{ab} \, .
    \end{aligned}
    \end{equation}
    In coordinates, $T$ is given by eq.~\eqref{plane_int_plane}, namely
    \begin{equation}\label{ab_bar_expansion}
        e_{ab} \mapsto e_{a-1,a} \langle a+1, b-1,b,b+1 \rangle  + e_{a,a+1} \langle a-1, b-1,b,b+1 \rangle  - e_{a-1,a+1} \langle a ,b-1,b,b+1 \rangle .
    \end{equation}
    Here $e_{ab} \in \mathbb{R}^{\binom{n}{2}} \cong \bigwedge^2 \mathbb{R}^n$ is identified with the wedge product of the standard basis vectors $e_a, e_b \in \mathbb{R}^n$. 
    The statement follows if $T(\bR_{>0}^{\binom{n}{2}})$ does not intersect the positive orthant $\bR_{>0}^{\binom{n}{2}}.$ For that, we show that there exists a hyperplane separating the two cones. 
    We choose $v = \sum_{1 \leq a<b \leq n} v_{ab} \, e_{ab} $ such that $v_{ab} >0$ if $|b-a|=2$ and equal to zero otherwise. Then, by eq.~\eqref{ab_bar_expansion} and the fact that $Z$ is positive, $T_{ab} \cdot v_{ab} \leq 0$ for every $1 \leq a<b \leq n$. This means that the hyperplane normal to $v$ indeed separates $T(\bR_{>0}^{\binom{n}{2}})$ and $\bR_{>0}^{\binom{n}{2}}.$
\end{proof}

\begin{lemma}\label{lemma:Gr_C_connected}
    The sets $\Gr(2,4) \cap \inter( C_{2,2,n})$ and $\Gr(2,4) \cap \inter( \widetilde{C}_{2,2,n})$ are both connected.
\end{lemma}

\begin{proof}
    Let us denote by $S$ one of the two sets in the statement, since the same argument works in both cases. Let $Y$ be a point in $ S$. Recall that by eq.~\eqref{eq:C_tilde_ineq} we have that $\langle Y \, \overline{ij} \rangle > 0$ for every $1 \leq i<j\leq n$. Then, by Lemma~\ref{lemma:exist_negative_ij} there exist $1 \leq i< j \leq n$ such that $\langle Y \, ij \rangle < 0$. We can can assume without loss of generality that $i$ and $j$ are different than $1$. We show that there exists a path in $S \cup \{(12)\}$ from $Y$ to $(12) \in \Gr(2,4)$. Consider the two-dimensional plane $P$ parameterized by $\alpha, \beta \in \mathbb{R}$ as $P(\alpha,\beta) := \alpha \, (12) +  (ij) + \beta \, Y \in \mathbb{P}(\bigwedge^2 \mathbb{R}^4)$. We intersect $P$ with $\Gr(2,4)$. Since points in $\Gr(2,4)$ correspond to decomposable vectors in $\mathbb{P}(\bigwedge^2 \mathbb{R}^4)$, we obtain the equation
    \begin{equation}
        P(\alpha,\beta) \wedge P(\alpha,\beta) = \alpha \, \langle 12 ij \rangle + \alpha \beta \, \langle Y  12 \rangle + \beta \, \langle Y  ij \rangle = 0 \, .
    \end{equation}
    This is the defining equation of a curve in the plane $P.$ We solve
    \begin{equation}\label{eq:alpha}
        \alpha(\beta) = \frac{- \beta \, \langle Y  ij \rangle }{ \langle 12 ij \rangle + \beta \, \langle Y  12 \rangle } \, ,
    \end{equation}
    and define $\gamma(\beta):=P(\alpha(\beta),\beta)$. Then $\gamma$ is a parameterization of the curve $P \cap \Gr(2,4)$. Note that $\langle Y 12 \rangle > 0$ and $\langle 12 ij \rangle >0$. Also, $\alpha , \beta > 0$ parameterize the (relative interior of the) convex hull of $(12)$, $(ij)$ and $Y$, which lies in $S$. Eq.~\eqref{eq:alpha} is in fact strictly positive for every $\beta > 0$, which means that $\gamma(\beta) $ is in $S$ for $\beta > 0$. Then $\gamma(\beta)$  traces a continuous path in $S $ as  $\beta$ varies in $(0,\infty).$ The endpoints are $\gamma(0) = (12)$ and $\lim_{\beta \rightarrow \infty} \gamma(\beta) = \lim_{\beta \rightarrow \infty} \frac{1}{\beta} P(\alpha(\beta),\beta) =   Y$. 

    Above we showed that the set $S \cup \{(12)\}$ is connected. We now show that $S$ is also connected.
    Let $P$ be either $C_{2,2,n}$ or $\widetilde{C}_{2,2,n}.$ Since $\Gr(2,4)$ is a manifold, we may take a local chart around $(12).$ We may choose the chart small enough such that the image of $S $ in the chart is a contractible set. Since $(12)$ lies on its boundary, $S$ is connected.
\end{proof}

We now summarize some known topological properties of the amplituhedron.
The amplituhedron $ \mathcal{A}_{k,m,n}$ is closed in $\Gr(k,k+m)$ by definition. It is regular since it is the continuous image of $\Gr_{\geq 0}(k,n)$, which is regular because it is homeomorphic to a
$k(m - k)$-dimensional closed ball~\cite{Galashin_2022}. By the same argument, the interior of $ \mathcal{A}_{k,m,n}$ is connected.
We are now fully equipped to prove the main result of this section.

\begin{theorem}[Intersection result for $k=m=2$]\label{thm:int_k=m=2} 
We have that
\begin{equation}
    \mathcal{A}_{2,2,n}(Z) = \Gr(2,4) \cap C_{2,2,n}(Z) =  \Gr(2,4) \cap \widetilde{C}_{2,2,n}(Z) \, , \qquad \forall \, Z \in {\rm Mat}_{>0}(4,n) \, .
\end{equation}
\end{theorem}

\begin{proof}
    This follows from Lemma~\ref{lemma:inters with polytope} by taking $X= \Gr(2,4)$, $\mathcal{A}= \mathcal{A}_{2,2,n}(Z)$ and $P=C_{2,2,n}(Z)$ or $P=\widetilde{C}_{2,2,n}(Z)$, respectively. As discussed above, the amplituhedron is a regular semialgebraic set and its relative interior is connected. Assumption~1 of Lemma~\ref{lemma:inters with polytope} follows from Lemma~\ref{leamma:reg_C_{2,2,n}}, assumption~2 from Lemma~\ref{lemma:Gr_C_connected}, assumption~3 is immediate and assumption~4 from~\cite[Proposition 3.1]{Ranestad_2024} and Proposition~\ref{prop:C_facets}.
\end{proof}

By Proposition~\ref{prop:amphull} we obtain the following result.
\begin{cor}[Convexity for $k=m=2$]\label{cor:m=k=2_conv_hull}
    The amplituhedron $\mathcal{A}_{2,2,n}(Z)$ is extendably convex for every $Z \in {\rm Mat}_{>0}(4,n)$.
\end{cor}

\section{Extendable dual amplituhedra}
\label{sec:Dual amplituhedra}

In this section we define a notion of convex duality for semialgebraic sets in a real projective variety and later apply it to define a notion of dual amplituhedron.
    
\subsection{Extendable duality}
\label{subsec:Extendable duality}

Through the following section, let $X \subset \mathbb{P}^N$ be a real projective variety.

\begin{definition}
    Let $S \subset X$ be a connected, semialgebraic, and very compact set. We define the \textit{extendable convex dual of $S$ in $X$} to be
    \begin{equation}
        S^{*_X} := X \cap S^{*} = X \cap \conv(S)^{*}  \, ,
    \end{equation}
    where $S^{*}$ denotes the convex dual of $S$ in $\mathbb{P}^N$.
\end{definition}
Note that since $S$ is very compact, so is its (projective) convex hull. By the Tarski-Seidenberg Principle~\cite[Section 1.4]{bochnak2010real} it follows that $\conv(S)^{*}$ is semialgebraic, and therefore so is $ S^{*_X}$. It is important to point out that the set $ S^{*_X}$ depends on the choice of isomorphism $(\mathbb{P}^N)^\vee \cong \mathbb{P}^N$, and therefore on the choice of inner product on $\mathbb{R}^{N+1}$. In the following, we consider the standard inner product.
We now prove some properties of this notion of convex duality, which follow directly from elementary properties of the usual convex hull and convex duality in projective space.

\begin{proposition}
    Let $S,T \subset X$ be connected, semialgebraic, and very compact sets. 
    \begin{enumerate}
        \item $S^{*_X}$ is convex in $X$.
        \item If $S \subset T$, then $T^{*_X} \subset S^{*_X}$.
        \item $\conv_X(S)  \subset (S^{*_X})^{*_X}$.
    \end{enumerate}
\end{proposition}

\begin{eg}
    Unlike in the projective case, if $S$ is convex in $X$, the inclusion $S  \subset (S^{*_X})^{*_X}$ can be strict. For example, in an affine chart on $\mathbb{P}^2$ let $X = \mathcal{V}(x^2 + y^2 -1)$ be the unit circle and $S = \{(x,y) : x^2 + y^2 \leq 1, \, x \geq 0 \}$ be a half-disk.
    Then, $S^* = \{(x,y) : x^2 + y^2 \leq 1, \, x \leq 0 \} \cup \{(x,y) : x \geq 0, -1 \leq y \leq 1\}$, and $S^{*_X} = X$. Thus the inclusion $ (S^{*_X})^{*_X} = X \supset S$ is strict.
\end{eg}
%Let us consider as an example a common semialgebraic set in the Grassmannian.
\begin{lemma}
    The non-negative Grassmannian is self-dual, that is,
    \begin{equation}
        \Gr_{\geq 0}(k,n)^{*_{\Gr(k,n)}} = \Gr_{\geq 0}(k,n) \, .
    \end{equation}
\end{lemma}

\begin{proof}
    This follows from Proposition~\ref{prop:amphull} with $n=k+m$ and $Z$ being the identity matrix, together with the fact that $C_{k,m,k+m}(Z)$ is the standard projective simplex, which is its own convex dual.
\end{proof}

\subsection{The extendable dual amplituhedron for $k=m=2$}
\label{subsec:The dual amplituhedron for $k=m=2$}

We now use the results from Subsection \ref{subsec:Extendable duality} to define a notion of dual amplituhedron.
\begin{definition}\label{def:extendable_dual}We define the \textit{extendable dual amplituhedron} to be
\begin{equation}
    \widetilde{\mathcal{A}}_{k,m,n} := \mathcal{A}_{k,m,n}^{*_{\Gr(k,k+m)}} = \Gr(k,k+m) \cap \conv(\mathcal{A}_{k,m,n})^{*} = \Gr(k,k+m)  \cap  C_{k,m,n}^{*} \, ,
\end{equation}
where in the last equality we used Proposition~\ref{prop:amphull}.
\end{definition}
Since $C_{k,m,n}$ is the convex hull of the $\binom{n}{k}$ points in~\eqref{vertices}, we have that
\begin{equation}
    C_{k,m,n}^{*} = \Big\{ Y \in \mathbb{P}(\bigwedge^{k} \mathbb{R}^{k+m}) \, : \, \langle Y \,  i_1 \dots i_k \rangle \geq 0 \, , \ \forall \, (i_1 \dots i_k) \in \binom{[n]}{k} \Big\} \, .
\end{equation}

Let us now focus on the case of $k=m=2$. By eq.~\eqref{eq_C_tildeC_duality} and Theorem~\ref{thm:int_k=m=2} we have that
\begin{equation}
\begin{aligned}
    \widetilde{\mathcal{A}}_{2,2,n}(Z) &= \Gr(2,4)  \cap  C_{2,2,n}(Z)^{*} = \Gr(2,4)  \cap  \widetilde{C}_{2,2,n}(W)  \\
    &= \Gr(2,4)  \cap  C_{2,2,n}(W) = \mathcal{A}_{2,2,n}(W) \, ,
\end{aligned}
\end{equation}
where we specified the dependence of the amplituhedron and of the polytopes on the positive matrix. In particular, for $k=m=2$ the extendable dual amplituhedron is again an amplituhedron. This also explains our notation for the dual amplituhedron.

We now argue that this notion of duality for the amplituhedron at $k=m=2$ is natural from a physics point of view. For that, it is useful to think of the amplituhedron $\mathcal{A}_{k,m,n}$ as living in $\Gr(m,n)$, by identifiying it with its image under the \textit{twistor embedding} $\psi \, \colon \, \Gr(k,k+m) \rightarrow \Gr(m,n)$, see \cite[Section 2]{Lam:2024gyg}. Let us denote the amplituhedron in the twistor embedding by $\mathcal{B}_{k,m,n} = \psi(\mathcal{A}_{k,m,n})$. By \cite[Theorem 4]{Lam:2024gyg}, we have that 
\begin{equation}
    \mathcal{B}_{k,2,n}(Z) = \overline{\Pi_{+,k}^{\circ}(Z) \cap \Gr(2,Z)} \, ,
\end{equation}
where $\Pi_{+,k}^{\circ}(Z)$ is the connected component of the complement of the $n$ divisors $\langle ii+1\rangle =0$ in $\Gr(2,Z)$ on which the sequence $(\langle 12 \rangle , \langle 13 \rangle , \dots , \langle 1n \rangle)$ has $k$ sign-flips; see \cite[Proposition 2]{Lam:2024gyg}. In this notation, we identify the image of the extendable dual amplituhedron under the twistor embedding with
\begin{equation}
    \psi(\widetilde{\mathcal{A}}_{2,2,n}(Z)) = \psi(\mathcal{A}_{2,2,n}(W)) = \mathcal{B}_{0,2,n}(Z) \, .
\end{equation}
From this we deduce that $\widetilde{\mathcal{A}}_{2,2,n}(Z)$ has the following semialgebraic description. It consists of all points $Y \in \Gr(2,4)$ such that
\begin{equation}\label{semialg_dual_ampl}
\begin{aligned}
    &\langle Y \, ii+1 \rangle > 0 \,  \quad \forall\, i=1,\dots,n-1 \, , \quad  \langle Y1n\rangle > 0 \, , \text{ and } \\
    &\text{the sequence } (\langle Y \, 12 \rangle,\langle Y \, 13 \rangle \dots , \langle Y \, 1n \rangle) \text{ has zero sign flips} .
\end{aligned}
\end{equation}
\begin{cor}[Extendable dual amplituhedron for $k=m=2$]\label{cor:dual_ampl}
    The extendable dual amplituhedron for $k=m=2$ in twistor space is equal to the twisted amplituhedron, which has a semialgebraic description as in~\eqref{semialg_dual_ampl}.
\end{cor}
Therefore, for $k=m=2$ the duality of Definition \ref{def:extendable_dual} corresponds physically to interchanging the maximal-helicity-violating (MHV) ($k=2$) sector with the $\overline{\text{MHV}}$ sector; see also~\cite[Section 2.3]{Herrmann_2021}.

We now comment on $m=2$ and $k=3$. In Example~\ref{eg:k=3_n=6}  we determined computationally the Schubert exterior polytope $\widetilde{C}_{3,2,6}$. We found that eq.~\eqref{eq_C_tildeC_duality} holds true for these values of the parameters, namely, that the dual of the twisted exterior cyclic polytope equals the Schubert exterior cyclic polytope. Therefore, if the analogue of the intersection result in Theorem~\ref{thm:int_k=m=2} is true also for this case, then the same argument presented in this section would go through, and the extendable dual amplituhedron $\widetilde{\mathcal{A}}_{3,2,6}$ would have the same semialgebraic description as in eq.~\eqref{semialg_dual_ampl} with $Y$ in $ \Gr(2,5)$. More generally, we are lead to the following question.

\begin{question}\label{question:ampl_duality}
    Is the analogue of Corollary~\ref{cor:dual_ampl} true for $m=2$ and every $k \geq 2$?
\end{question}

\section{Open problems}

Previously we stated Questions~\ref{question:C_duality},~\ref{question_int} and~\ref{question:ampl_duality} as open problems. We now present additional problems and questions related to extendable convexity, duality, and exterior cyclic polytopes. 

\begin{question}[Combinatorial Type Stratification]
     As we have shown in Section~\ref{sec:Exterior Cyclic Polytopes}, the combinatorial type of the exterior cyclic polytope $C_{k,m,n}(Z)$ changes as $Z$ varies over ${\rm Mat}_{>0}(k+m,n)$. It would be interesting to study the closed set of positive $Z$ where the matroid of $\wedge^k Z$ differs from $W_{k,m,n}$, and to understand how the combinatorial type of $C_{k,m,n}(Z)$ changes when crossing this locus. This would presumably involve understanding which bases of $W_{k,m,n}$ correspond to polynomials that are positive on the positive Grassmannian $\Gr_{ >0}(k+m,n)$, as in Theorem~\ref{thm:combotype}. 
\end{question}

\begin{question}[Triangulations and Tiles]
    Recent advances in the understanding of tiles and tilings of the amplituhedron have been made for $m=1,2,4$~\cite{Karp_2017,Parisi:2021oql,Even-Zohar:2023del}, respectively. It would be interesting to study triangulations of the exterior cyclic polytope $C_{k,m,n}(Z)$ and determine the interplay between the two. We propose the case of $k=m=2$ as a good starting point for this analysis, due to the complete understanding of the tiles of the amplituhedron~\cite{Parisi:2021oql} and our Theorem~\ref{thm:int_k=m=2}. %We observe that in this case the intersections of certain simplices in $C_{2,2,n}(Z)$ with $\Gr(2,4)$ yield tiles for $\mathcal{A}_{2,2,n}(Z)$. However, there are tiles of $\mathcal{A}_{2,2,n}(Z)$ which arise from the intersection with bigger sub-polytopes in~$C_{2,2,n}(Z)$.
\end{question}

%\begin{question}[Extendable convexity of Grasstopes]
%     Generalizing the amplituhedron construction, one defines a \textit{Grassmann polytope}~\cite{Lam:2014omj} as the image of a positroid cell under \eqref{Z_tilde_map}, and a~\textit{Grasstope} as the image of the non-negative Grassmannian under~\eqref{Z_tilde_map} without the assumption on $Z$ being positive \cite{Arkani_Hamed_2017,mandelshtam2025combinatoricsm1grasstopes}. It would be interesting to study the extendable convex hull of Grassmann polytopes and Grasstopes, as well as their convexity properties in the Grassmannian.
%\end{question}

\begin{question}[Adjoints]
    In \cite{Ranestad_2024} the authors prove that $\mathcal{A}_{2,2,n}(Z)$ is a positive geometry  for every $n \geq 4$ by determining its adjoint hypersurface in $\Gr(2,4)$ and showing that it is unique. Since $C_{2,2,n}(Z) $ and $\widetilde{C}_{2,2,n}(Z) $ are convex polytopes, each of them possesses a unique adjoint hypersurface in $\mathbb{P}^5$. It would be interesting to determine the relationship between these adjoint hypersurfaces. %It would be interesting if this argument could be extended to higher parameters~$m$ and~$k$, and used to prove that the amplituhedron is a positive geometry.
\end{question}

\begin{question}[Canonical Forms and Dual Volumes]
    %Theorem~\ref{thm:int_k=m=2} together with the comment in the previous question hint to the fact that the amplituhedron $\mathcal{A}_{2,2,n}$ is positively convex, according to the definition in~\cite[Section 9]{Arkani_Hamed_2017} and~\cite[Section 3.1]{Kohn_2025}. 
    Finding the \textit{dual amplituhedron} is a long-standing problem. In particular, one would like to have a representation of the canonical form of the amplituhedron as an integral over some dual geometry. In this paper we proposed a candidate for the underlying dual semialgebraic set in Section~\ref{sec:Dual amplituhedra}. However, we do not have yet an integral representation. We believe that the missing ingredient is a non-negative measure, supported on $\widetilde{A}_{2,2,n}$. In~\cite{progress} we plan to come back to the study of measures representing canonical forms of full-dimensional semialgebraic sets in projective space.
\end{question}

\section{Acknowledgements}
We would like to thank Bernd Sturmfels for many helpful conversations and for sharing some computations. We also thank Rainer Sinn, Steven Karp, Matt Larson, Zachary Greenberg, Dani Kaufman, Johannes Henn and Matteo Parisi for helpful discussions. % also zachary, dani, joris, ...
EM is funded by the European Union (ERC, UNIVERSE PLUS, 101118787). Views and opinions expressed are those of the authors only and do not necessarily reflect those of the European Union or the European Research Council Executive Agency. Neither the European Union nor the granting authority can be held responsible for them. EP is funded by NSF GRFP no. 2023358166.

\printbibliography

\appendix
\section{Poset of bases of $W_{2,2,n}$}\label{app:hasse}
This section contains supporting computations for the proof of Theorem \ref{thm:combotype}. There are $47$ combinatorial types of bases in the matroid $W_{2,2,n}.$ They form a poset under cutting and gluing as discussed in Subsection \ref{subsec:The case $m=k=2$}. This poset is pictured in Figure \ref{fig:hasse}. The bases are pictured, in the order of the Hasse diagram, in Figure \ref{fig:w22nbases}. The graphs at the top of the poset are $K_4$ (number $1$), the house-shaped graph (number $3$), and the cycle (number $12$).

\begin{figure}[!h]
    \centering
\includegraphics[width=0.8\linewidth]{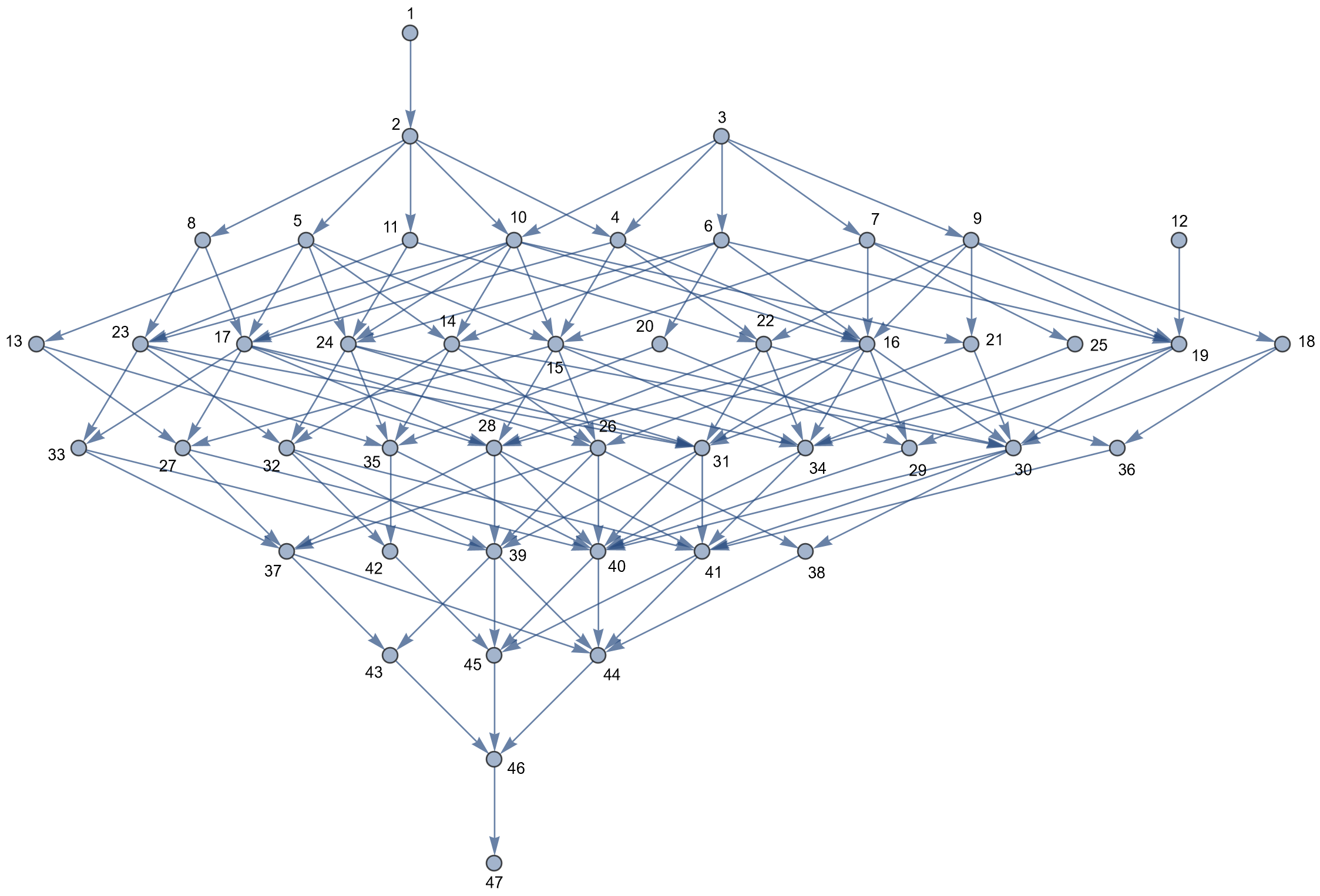}
    \caption{The poset of bases of $W_{2,2,n}$ under cutting}
    \label{fig:hasse}
\end{figure}

\begin{figure}[htbp]
    \centering
    \includegraphics[width=0.9\textwidth]{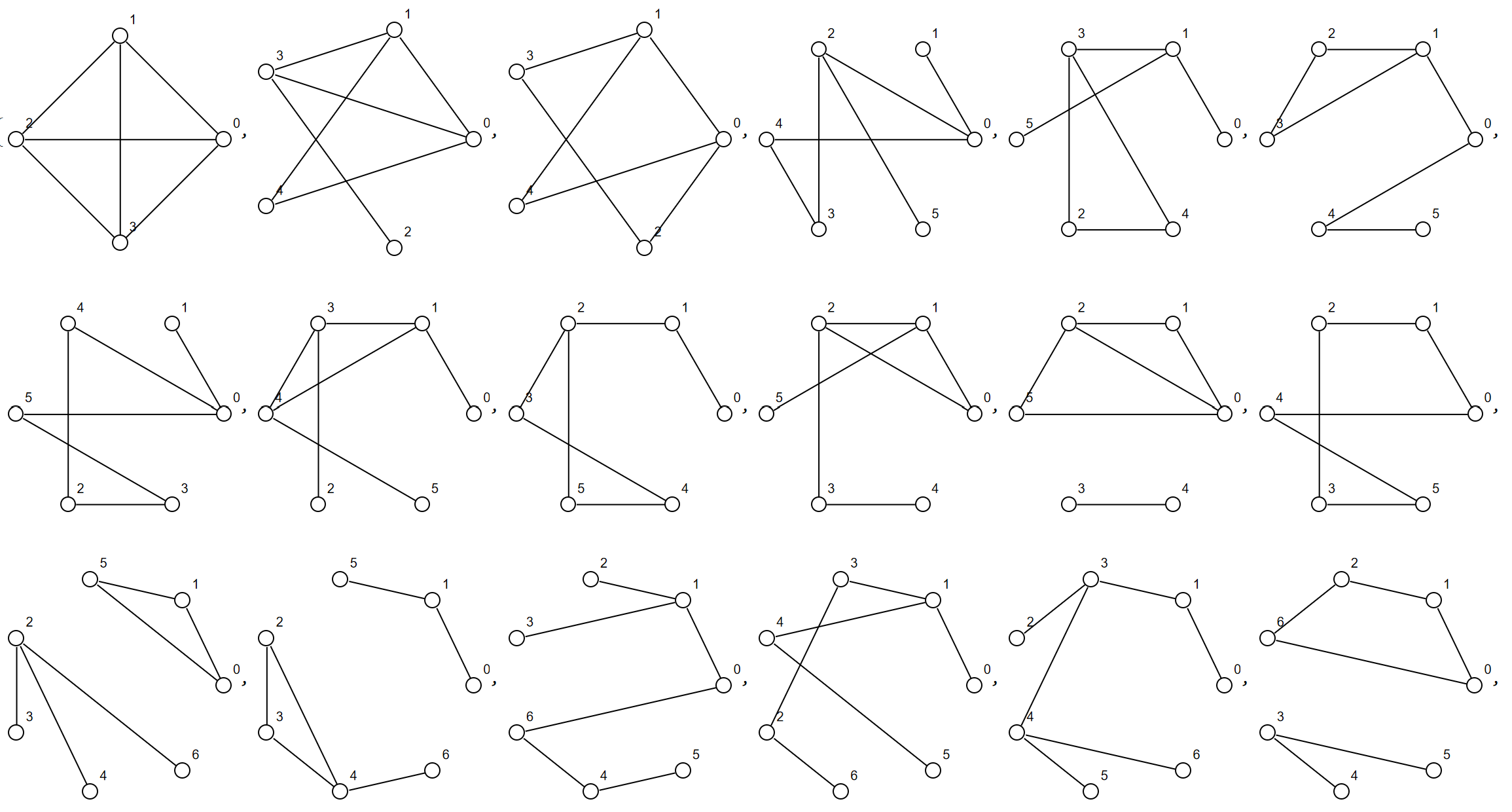}

\includegraphics[width=0.9\textwidth]{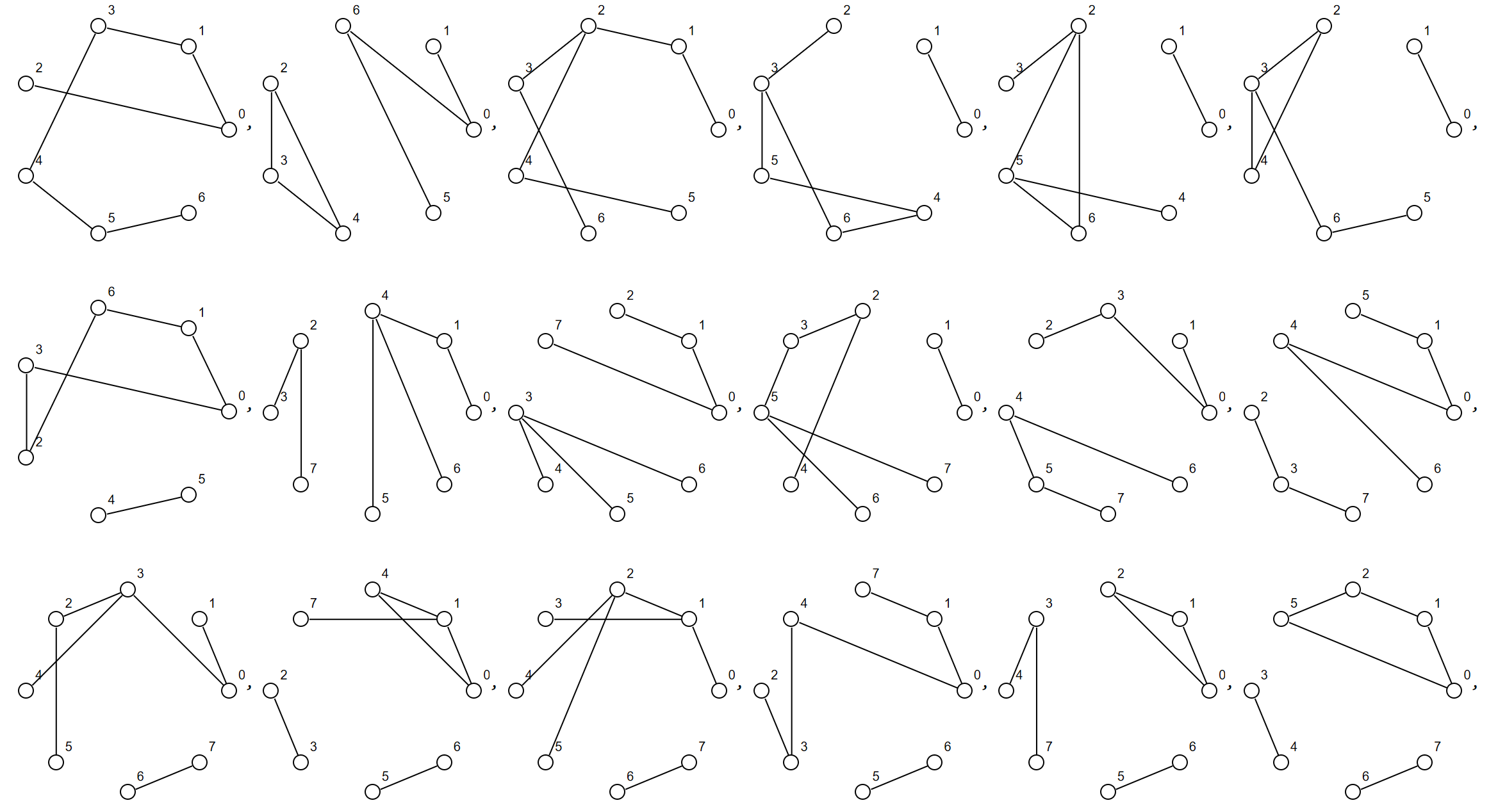}
    \includegraphics[width=0.9\textwidth]{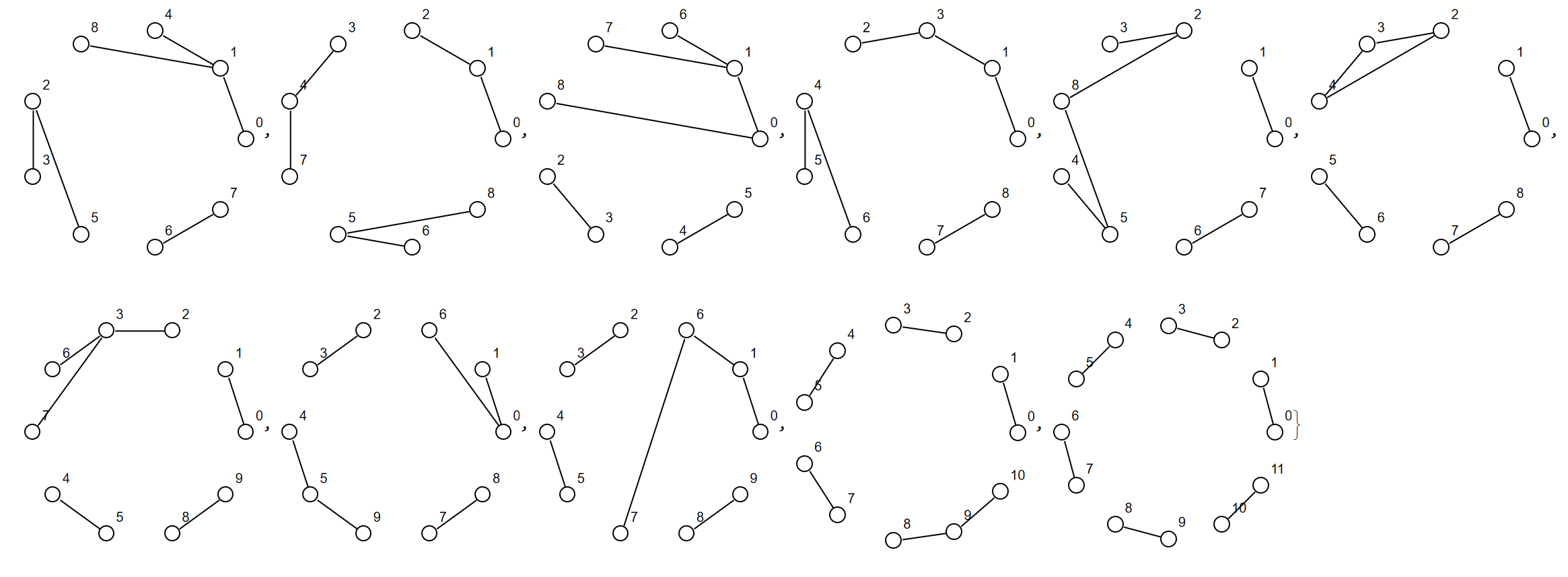}
    \caption{Bases of $W_{2,2,n}$}\label{fig:w22nbases}

\end{figure}

\section{Circuits of $W_{2,3,n}$}
\label{app:Computations for W_{2,3,n}}
In this section we list out  circuits of the exterior power matroid $W_{2,3,n},$ up to symmetry. Recall from Lemma \ref{lem:gluing} that gluing vertices together produces a dependent set. We do not list all of the circuits which may be obtained from other circuits by gluing together along vertices. 
With these conventions, the circuits of $W_{2,3,n}$ are listed in Figure \ref{fig:w25ncircuits}.

\begin{figure}[!h]
    \centering
\includegraphics[width=0.7\linewidth]{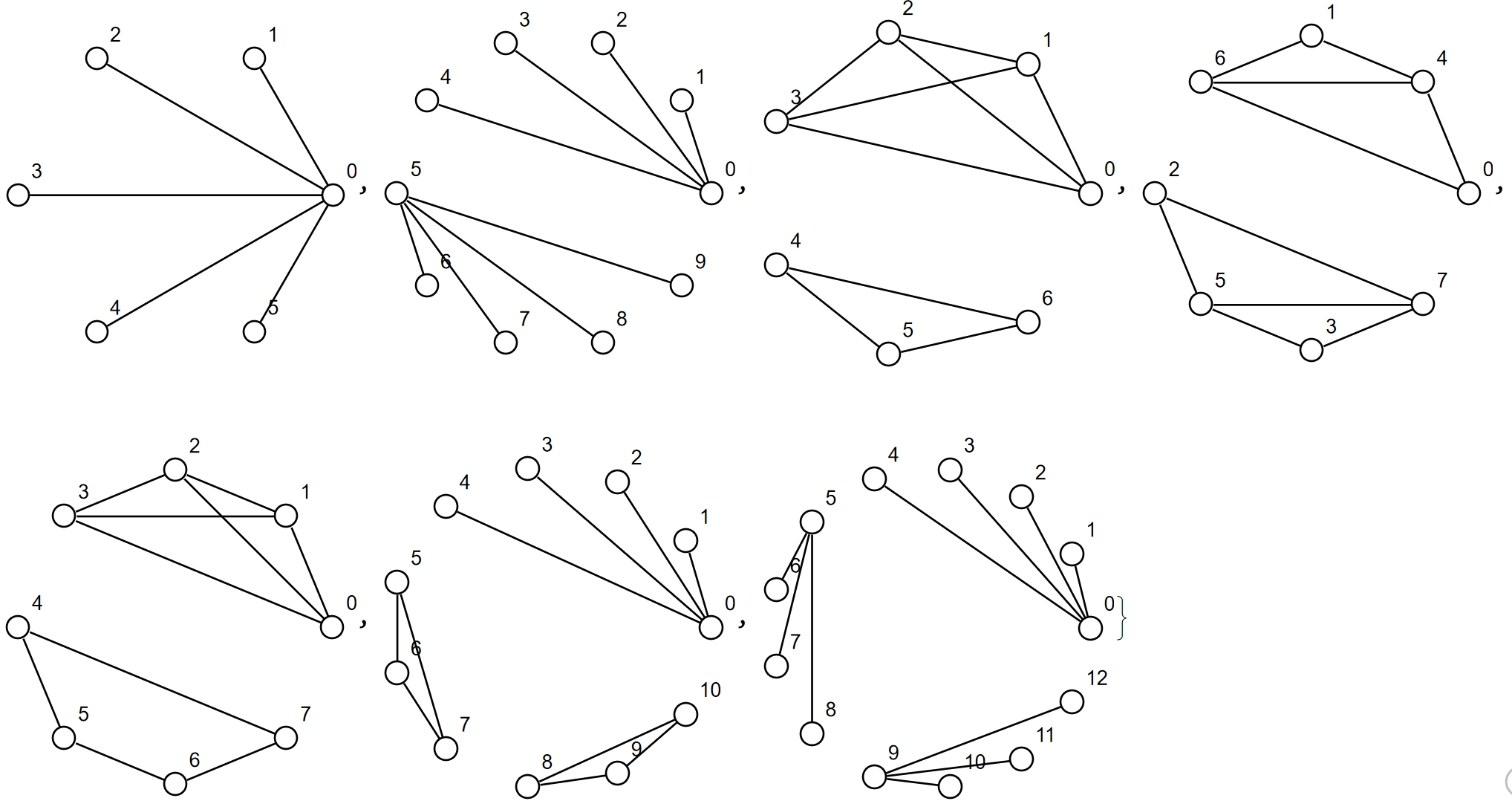}
    \caption{Circuits of $W_{2,3,n}$ up to gluing}
    \label{fig:w25ncircuits}
\end{figure}
One may show that these are circuits by finding linear forms which vanish on them. To this end, we compute for each circuit $C$ a Schubert variety (or an intersection of Schubert varieties) $\Omega$ such that $\text{span}\{Z_i \wedge Z_j \ : \ ij \in C\} \cap \Gr(k,k+m)$ is a subset of $ \Omega.$ Linear forms which vanish on $\Omega$ will then also vanish on $C,$ because $C$ is a linear space. For example, in $\Gr(2,5)$ the three linear forms $p_{12}, p_{13}, p_{23}$ vanish on the Schubert variety of lines meeting the line $\text{span}\{e_4, e_5\}.$ The Schubert varieties are as follows, from top to bottom and left to right in each row.

\begin{enumerate}[label=\roman*.]
		\item Lines in $\bP^4$ meeting the point $0$ (codimension $6$).
		\item Lines in $\bP^4$ meeting the line $05$ (codimension $3$).
		\item Lines meeting any two of the planes $(3456) \cap (0123), \, (2456) \cap (0123), (1456) \cap (0123),$ and $(0456) \cap (0123)$ -- note that meeting two implies meeting the other two, so this is only two Schubert conditions. Also note that $(3456) \cap (0123)$ is a plane in $(0123),$ so all lines in $(0123)$ meet it. 
		\item Lines in $\bP^4$ meeting the plane $(0146) \cap (2357)$.
		\item Lines in $\bP^4$ meeting the plane $(0123) \cap (4567)$.
		\item Lines in $\bP^4$ meeting the plane $(0567) \cap (089,10)$.
		\item Lines in $\bP^4$ meeting the plane $(059)$.
	\end{enumerate}

Any hyperplane $F$ containing a circuit other than the $5$-star is a Schubert hyperplane. We certify this using Lemma \ref{lem:schuberthyp} by producing a transversal plane to $F$ in each case. One may check that the conditions listed in each bullet point give a codimension $6$ condition on $\Gr(2,5),$ so that such a transversal plane indeed exists.

\begin{enumerate}[label=\roman*.]
    \item Containing the five-star does not imply that a hyperplane $F$ is Schubert.
    \item Here $F$ can be spanned by $7$ lines from the circuit and $2$ other lines. The transversal is the plane meeting the points $0$ and $5$ and two other lines.
    \item Here $F$ can be spanned by $8$ lines from the circuit and $1$ other line. Plane contained in $(0123)$ and containing the  line $(456) \cap (0123)$ and meeting one line.
    \item Here (and for the other circuits with $10$ elements) $F$ can be spanned by $9$ elements of the circuit. So, the transversal plane is exactly the plane described in the previous list.
\end{enumerate}

\printindex
\end{document}